\documentclass[preprint,12pt,3p]{elsarticle}

\usepackage{amssymb}
\usepackage{amsmath,amsthm}
\usepackage{mathrsfs}
\usepackage{hyperref}
\usepackage{cleveref}
\usepackage[all,cmtip]{xy}
\usepackage{epsf, graphicx}
\usepackage{latexsym,amsfonts,amsbsy,amssymb}
\usepackage{geometry}
\usepackage{titletoc}
\usepackage{color,xcolor}
\usepackage{rotating}
\usepackage{algorithm}
\usepackage{algorithmic}
\usepackage{amscd}

\makeatletter

\@addtoreset{equation}{section} \makeatother
\newtheorem{theorem}{Theorem}[section]
\newtheorem{lemma}[theorem]{Lemma}
\newtheorem{notation}[theorem]{Notation}
\newtheorem{conjecture}[theorem]{Conjecture}
\newtheorem{corollary}[theorem]{Corollary}
\newtheorem{remark}[theorem]{Remark}
\newtheorem{void}[theorem]{}
\newtheorem{example}[theorem]{Example}
\newtheorem{definition}[theorem]{Definition}
\newtheorem{proposition}[theorem]{Proposition}

\def\br{{\rm br}}

\def\Aut{{\rm Aut}}

\def\Hom{{\rm Hom}}

\def\End{{\rm End}}
\def\Gal{{\rm Gal}}
\def\Ob{{\rm Ob}}
\def\Mor{{\rm Mor}}

\def\C{\mathcal{C}}
\def\Id{{\rm Id}}
\def\G{\mathcal{G}}

\def\O{\mathcal{O}}

\def\bF{\overline{\F}}
\def\S{\mathcal{S}}
\def\E{\mathcal{E}}
\def\X{\mathcal{X}}
\def\F{\mathbb{F}}
\def\FF{\mathcal{F}}

\def\T{\mathcal{T}}

\def\W{\mathcal{W}}
\def\I{\mathcal{I}}

\def\Iso{{\rm Iso}}

\makeatletter
\def\ps@pprintTitle{%
\let\@oddhead\@empty
\let\@evenhead\@empty
\def\@oddfoot{\reset@font\hfil\thepage\hfil}
\let\@evenfoot\@oddfoot
}
\makeatother

\begin{document}

\begin{frontmatter}

\title{The Galois Alperin weight conjecture for finite category algebras}

\author{Xin Huang}


\begin{abstract}
Let $p$ be a prime, $k$ an algebraic closure of $\mathbb{F}_p$ and $\Gamma$ the Galois group ${\rm Gal}(k/\mathbb{F}_p)$. Let $\mathcal{C}$ be a finite category and $\mathcal{O}_{\mathcal{C}}$ the $p$-orbit category of $\mathcal{C}$ defined by Linckelmann \cite{Lin14}. We formulate a version of the Galois Alperin weight conjecture (GAWC) for finite category algebras stating that there exists a $\Gamma\times {\rm Aut}(\mathcal{C})$-equivariant bijection between the set of isomorphism classes of simple $k\mathcal{C}$-modules and that of the weights of $k\mathcal{O}_{\mathcal{C}}$. We reduce the GAWC for finite categories to finite groups. 
For $\mathcal{C}$ an EI-category, we give a partition of weights of $k\mathcal{O}_{\mathcal{C}}$ with respect to blocks of $k\mathcal{C}$ and then formulate a blockwise Galois Alperin weight conjecture (BGAWC) for $\mathcal{C}$. Similarly, we reduce the BGAWC for finite EI-categories to finite groups.

\end{abstract}


\end{frontmatter}


\section{Introduction}\label{s1}

The (blockwise) Galois Alperin weight conjecture, due to Navarro \cite{Nav04}, a strong form of Alperin's weight conjecture \cite{Alp}, predicts that the fields of values of irreducible Brauer characters (of a block) of a finite group can be locally determined. In this paper we extend the (blockwise) Galois Alperin weight conjecture from finite groups to finite categories. 

Throughout this paper we fix a prime number $p$ and an algebraic closure $k$ of $\F_p$ and denote by $\Gamma$ the Galois group $\Gal(k/\F_p)$. A category $\C$ is called {\it finite} if its morphism class is a finite set. Any finite group can be viewed as a finite category with one object.

 Let $\C$ be a finite category. According to \cite[Definition 1.1]{Lin14}, the {\it $p$-transporter category of $\C$} is the finite category $\T_\C$
defined as follows. The objects of $\T_\C$ are the pairs $(X,P)$ consisting of an object $X$ of $\C$ and a not
necessarily unitary $p$-subgroup $P$ of the monoid $\End_\C(X)$. For any two objects $(X,P)$ and $(Y,Q)$, the
morphism set $\Hom_{\T_\C}((X,P), (Y,Q))$ is the set of all triples $(s, P, Q)$ where $s: X\to Y$ is a morphism
in $\C$ satisfying $s = s\circ 1_P = 1_Q \circ s$ and $s \circ P \subseteq Q \circ s$. The composition of morphisms in $\T_\C$ is induced
by that in $\C$.  The identity morphism of an object $(X,P)$ in $\T_\C$ is $(1_P,P,P)$. If no confusion arise, we will denote a morphism $(s,P,Q)$ in $\T_\C$ again by $s$. Allowing nonunitary subgroups $P$ of $\End_\C(X)$ in the definition of objects of $\T_\C$ means that the unit element $1_P$ of $P$ need not be equal to ${\rm Id}_X$ but can be any idempotent endomorphism of $X$. The condition $s=s\circ 1_P=1_Q\circ s$ in this definition implies that $\Hom_{\T_\C}((X,P),(Y,Q))$ can be identified to a subset of $1_Q\circ \Hom_\C(X,Y)\circ 1_P$. With this identification, the morphism set $\Hom_{\T_\C}((X,P),(Y,Q))$ is a $Q$-$P$-subbiset of $1_Q\circ \Hom_\C(X,Y)\circ 1_P$ with respect to the actions induced by precomposing with morphisms in $P$ and composing with morphisms in $Q$. The condition $s\circ P\subseteq Q\circ s$ in the above implies that $Q\circ s\circ P=Q\circ s$; that is, any $Q$-$P$-orbit in $\Hom_{\T_\C}((X,P),(Y,Q))$ is in fact a $Q$-orbit. According to \cite[Definition 1.2]{Lin14}, the {\it $p$-orbit category of $\C$} is the finite category $\O_\C$ defined
as follows. The objects of $\O_\C$ are the same as those of $\T_\C$. For any two objects $(X,P)$ and $(Y,Q)$ of $\O_\C$,
the morphism set $\Hom_{\O_\C}((X,P), (Y,Q))$ is the set of all $Q$-$P$-orbits $Q\backslash\Hom_{\T_\C}((X,P), (Y,Q))/P$ of
morphisms in $\T_\C$. The composition of morphisms in $\O_\C$ is induced by that in $\T_\C$.

For a finite-dimensional $k$-algebra $A$, denote by $\S(A)$ the set of isomorphism classes of simple $A$-modules. 
Following \cite{LS}, given a finite category $\C$, the isomorphism classes of simple
$k\C$-modules are parametrised by isomorphism classes of pairs $(e, T)$, with $e$ an idempotent endomorphism
of some object $X$ in $\C$ and $T$ a simple $kG_e$-module, where $G_e$ is the group of all invertible
elements in the monoid $e\circ \End_\C(X) \circ e$. Such a pair $(e, T)$ is called a {\it weight} of $k\C$ if the simple $kG_e$-module $T$ is in addition projective. Denote by $\W(k\C)$ the set of isomorphism classes of weights of $k\C$; 
we review this in Section \ref{s3:weights of O and proof of main} below. The group-theoretic version of Alperin’s weight conjecture in \cite{Alp} 
is equivalent to the following statement:

\begin{conjecture}[Alperin~{\cite{Alp}}]\label{conj:AWC for finite groups}
For any finite group $G$, there exists a bijection between $\S(kG)$ and $\W(k\O_{G})$.
\end{conjecture}

 Linckelmann \cite{Lin14} extended Conjecture \ref{conj:AWC for finite groups} to finite categories:

\begin{conjecture}[Linckelmann~{\cite[Conjecture 1.3]{Lin14}}]\label{conj:AWC for category algebras}
For any finite category $\C$, there exists a bijection between $\S(k\C)$ and $\W(k\O_{\C})$.
\end{conjecture}

In \cite[Corollary 1.5]{Lin14}, Linckelmann showed that Conjectures \ref{conj:AWC for category algebras} and \ref{conj:AWC for finite groups} are in fact equivalent. To consider the action of the Galois group $\Gamma$, we need the following notation:

\begin{notation}\label{nation:Galois twist}
{\rm Let $A$ be a finite-dimensional $\F_p$-algebra and let $A':=k\otimes_{\F_p} A$. Let $\sigma\in \Gamma$. We see that $\sigma$ induces an $\F_p$-algebra automorphism  of $A'$ sending $x\otimes a$ to $\sigma(x)\otimes a$ for any $x\in k$ and $a\in A$. We abusively use the same symbol $\sigma$ to denote this $\F_p$-algebra automorphism. Clearly $\Gamma$ permutes the set of (primitive) central idempotents of $A'$. For any $A'$-module $U$, denote by ${}^{\sigma}U$ the $A'$-module which is equal to $U$ as a $k$-module endowed with the structure homomorphism $A'\xrightarrow{\sigma^{-1}} A'\to \End_k(U)$. Sometimes we will also write ${}^\sigma U$ as ${}_\sigma U$. Note that if $f:U\to V$ is an $A'$-module homomorphism, then $f$ is also an $A'$-module homomorphism ${}^\sigma U\to {}^\sigma V$. One checks that the map sending $x\otimes a$ to $\sigma^{-1}(x)\otimes a$ induces an isomorphism of $A'$-modules $A'\cong {}^\sigma A'$, where $x\in k'$ and $a\in A$. Now it is clear that ${}^\sigma U$ is a simple (resp. projective) $A'$-module if and only if $U$ is simple (resp. projective). Hence the group $\Gamma$ acts on the set $\S(A')$. Moreover, if $b$ is a central idempotent of $A'$, then $\Gamma_b$ (the stabiliser of $b$ in $\Gamma$) acts on the set $\S(A'b)$. By Proposition \ref{prop: Galois action on pairs} below, if $\C$ is a finite category and $\alpha\in Z^2(\C;\F_p^\times)$, then $\Gamma$ also acts on the set of (isomorphism classes of) weights of $k_\alpha\C$.
}
\end{notation}


In \cite{Nav04}, Navarro predicted a Galois refinement of the Alperin weight conjecture for finite group which can be reformulated as follows:

\begin{conjecture}[Navarro~{\cite{Nav04}; cf. \cite[Conjecture]{Turull14}}]\label{conj:GAWC for finite groups}
For any finite group $G$, there exists a bijection $\S(kG)\to \W(k\O_G)$ commuting with the action of $\Gamma$.
\end{conjecture}	

This leads to the obvious extension to finite categories:	
	
\begin{conjecture}\label{conj:GAWC for category algebras}
For any finite category $\C$, there exists a bijection $\S(k\C)\to \W(k\O_\C)$ commuting with the action of $\Gamma$.
\end{conjecture}		

If we forget the action of $\Gamma$, then Conjecture \ref{conj:GAWC for category algebras} returns to Linckelmann's Conjecture \ref{conj:AWC for category algebras}. Similar to Linckelmann's result \cite[Theorem 1.4]{Lin14}, we show that Conjectures \ref{conj:GAWC for finite groups} and \ref{conj:GAWC for category algebras} are in fact equivalent. This equivalence holds more generally for twisted group algebras and twisted category algebras with a $2$-cocycle in $Z^2(\C;\F_p^\times)$ which is extendible to the orbit category. There are canonical functors from $\C$ to $\T_\C$ and $\O_\C$ sending an object $X$ in $\C$ to $(X, \{{\rm Id}_X\})$, and a
morphism in $\C$ to its obvious images in $\T_\C$ and $\O_\C$, respectively. In particular, every $2$-cocycle $\alpha$ in
$Z^2(\O_\C;\F_p^\times)$ restricts to a $2$-cocycle in $Z^2(\C;\F_p^\times)$, again denoted by $\alpha$.

\begin{theorem}\label{theorem:main}
Let $\C$ be a finite category and $\alpha\in Z^2(\O_\C;\F_p^\times)$. If for any idempotent endomorphism $e$ in $\C$, there is a $\Gamma$-equivariant bijection $\S(k_\alpha G_e)\to \W(k_\alpha \O_{G_e})$, then there exists a $\Gamma$-equivariant bijection $\S(k_\alpha \C)\to \W(k_\alpha\O_{\C})$.
\end{theorem}

This will be proved in Section \ref{s3:weights of O and proof of main}. 

\begin{corollary}
Conjectures \ref{conj:GAWC for finite groups} and \ref{conj:GAWC for category algebras} are equivalent.
\end{corollary}

\begin{remark}
{\rm (i) In Theorem \ref{theorem:main}, the reason why we require $\alpha$ to be a $2$-cocycle with coefficients in $\F_p^\times$ is that in this case we have $k_\alpha \C \cong k\otimes_{\F_p} (\F_p)_{\alpha}\C$, and then the Galois group $\Gamma$ has an action on $\S(k_\alpha \C)$ as described in Notation \ref{nation:Galois twist}.

(ii) Note that the formulation in Theorem \ref{theorem:main} for twisted category algebras requires the $2$-cocycle $\alpha$ to be the restriction to $\C$ of a $2$-cocycle of $\O_\C$ along the canonical functor $\C\to \O_C$. It is not clear whether the map $H^2(\O_\C;\F_p^\times)\to H^2(\C;\F_p^\times)$ induced by the canonical functor $\C\to \O_\C$ is injective or surjective in general. The corresponding canonical functor $\C\to \T_\C$ sending $X$ to $(X,\{{\rm Id}_X\})$ poses no problem:
}
\end{remark}

\begin{proposition}[cf. {\cite[Proposition 1.6]{Lin14}}]
Let $\C$ be a finite category. The canonical functor $\C\to \T_\C$ induces a graded isomorphism $H^*(\T_\C;\F_p^\times)\cong H^*(\C;\F_p^\times)$.
\end{proposition}

As Linckelmann mentioned in \cite{Lin14}, it is less clear what happens under the functor $\T_\C\to \O_\C$ in general. We have the following special case for EI-categories (i.e. categories in which all endomorphisms are isomorphisms):

\begin{proposition}[cf. {\cite[Proposition 1.7]{Lin14}}]
Let $\C$ be a finite EI-category. Then $\T_\C$ and $\O_\C$ are EI-categories, and the canonical functor $\C\to\O_\C$ induces a graded isomorphism $H^*(\O_\C;\F_p^\times)\cong H^*(\C;\F_p^\times)$.	 
\end{proposition} 

One checks that the proofs of \cite[Propositions 1.6, 1.7, 4.6]{Lin14} did not use the blanket assumption there that the ground field is algebraically closed - we can replace the coefficient field there by any perfect field of characteristic $p$.

\begin{example}
{\rm Brauer algebras, Temperley--Lieb algebras, partition algebras, and their cyclotomic analogues can be interpreted as twisted monoid algebras (cf. \cite{Wilcox}). The $2$-cocycles of the underlying	monoids for these algebras satisfy the hypotheses of \cite[Proposition 4.6]{Lin14}; in particular, they are constant on maximal subgroups (so that their restrictions to maximal subgroups represent the trivial classes) and they extend to the associated orbit categories. Using a recent result in \cite{DH} that the BGAWC holds for symmetric groups, it is easy to see that the maximal subgroups of the underlying diagram monoids satisfy the GAWC, and hence so do these diagram algebras.} 
\end{example}

Since there is a blockwise Galois Alperin weight conjecture (BGAWC) for finite group, one will ask whether there is a BGAWC for a finite category $\C$. For this question we need to first give a partition of the weights of the orbit category algebra $k\O_\C$ with respect to blocks of $k\C$. Unfortunately, for the general case we didn't find such a partition. But if $\C$ is an EI-category, we can give such a partition; see Definition \ref{defi:partition of weights by blocks} below. For any central idempotent $b$ of a finite EI-category $\C$, we define the notion of a $b$-weight of $k\O_\C$. If $b=1$, then $\W(k\O_\C,b)$ returns to $\W(k\O_\C)$. We can easily show that $\Gamma_b$ acts on $\W(k\O_\C,b)$; see Definition \ref{defi:partition of weights by blocks}.

A finite group can be regarded as a finite EI-category, and the BGAWC for finite groups can be reformulated as follows:

\begin{conjecture}[Navarro~{\cite{Nav04}; cf. \cite[Conjecture 2]{H24}}]\label{conj:BGAWC for finite groups}
For any finite group $G$ and any central idempotent $b$ of $kG$, there exists a bijection $\S(kGb)\to \W(k\O_G,b)$ commuting with the action of $\Gamma_b$.
\end{conjecture}	

This leads to the obvious extension to finite EI-categories:	

\begin{conjecture}\label{conj:BGAWC for category algebras}
	For any finite EI-category $\C$ and any central idempotent $b$ of $k\C$, there exists a bijection $\S(k\C b)\to \W(k\O_\C,b)$ commuting with the action of $\Gamma_b$.
\end{conjecture}		

If we take $b=1$, then Conjecture \ref{conj:BGAWC for category algebras} returns to Conjecture \ref{conj:GAWC for category algebras}.
We show that Conjectures \ref{conj:BGAWC for finite groups} and \ref{conj:BGAWC for category algebras} are in fact equivalent. As in the block-free case, this equivalence holds more generally for twisted group algebras and twisted category algebras:

\begin{theorem}\label{theorem:main2}
	Let $\C$ be a finite EI-category, $b$ a central idempotent of $k\C$, and $\alpha\in Z^2(\O_\C;\F_p^\times)$. Then for any idempotent endomorphism $e$ in $\C$, setting $\hat{e}=\alpha(e,e)^{-1}e$, $\hat{e}b$ is a central idempotent in $k_\alpha G_e$; see Proposition \ref{prop:eb is in Ge} below. If for any idempotent endomorphism $e$ in $\C$, there is a $\Gamma_b$-equivariant bijection $\S(k_\alpha G_e \hat{e}b)\to \W(k_\alpha \O_{G_e},\hat{e}b)$, then there exists a $\Gamma_b$-equivariant bijection $\S(k_\alpha \C b)\to \W(k_\alpha\O_{\C},b)$. 
\end{theorem}

This will be proved in Section \ref{s4:Partition of weights} after we investigate the Brauer construction on twisted group algebras in Section \ref{s:Brauer construction}.

\begin{corollary}
	Conjectures \ref{conj:BGAWC for finite groups} and \ref{conj:BGAWC for category algebras} are equivalent.
\end{corollary}
    
Let $\C$ be finite category. Recall that a functor $\FF:\C\to \C$ is called an {\it automorphism} of $\C$ if there is another functor $\G:\C\to \C$ such that $\FF\circ \G={\rm Id}_\C=\G\circ \FF$. Let $\Aut(\C)$ denote the group of automorphisms of the category $\C$. Clearly if $\C$ is a finite group $G$, then $\Aut(\C)=\Aut(G)$. We can formulate analogues of Conjectures \ref{conj:AWC for category algebras} and \ref{conj:BGAWC for category algebras} with $\Aut(\C)$ instead of $\Gamma$. We need the following notation:

\begin{notation}\label{nation:automorphisms of category}
	{\rm Let $\FF$ be an automorphism of a finite category $\C$. Then for any objects $X$ and $Y$ in $\C$, $\FF$ induces an isomorphism $\FF_{X,Y}:\Hom_\C(X,Y)\cong \Hom_\C(\FF(X),\FF(Y))$; for shorthand we abusively denote the isomorphism $\FF_{X,Y}$ by $\FF$. As $X$ and $Y$ run over all objects in $\C$, those isomorphisms $\FF_{X,Y}$ induce a $k$-algebra automorphism of $k\C$ sending any $s\in \Hom_\C(X,Y)$ to $\FF(s)$; we still use the symbol $\FF$ to denote this automorphism of $k\C$.  Clearly $\FF$ permutes the set of (primitive) central idempotents of $k\C$. For any $k\C$-module $U$, denote by ${}_{\FF} U$ the $k\C$-module which is equal to $U$ as a $k$-module endowed with the structure homomorphism $k\C\xrightarrow{\FF^{-1}} k\C\to \End_k(U)$. Note that if $f:U\to V$ is an $k\C$-module homomorphism, then $f$ is also a $k\C$-module homomorphism ${}_{\FF} U\to {}_{\FF} V$. One checks that the map sending $a$ to $\FF^{-1}(a)$ induces an isomorphism of $k\C$-modules $k\C\cong {}_{\FF} k\C$, where $a\in k\C$. Now it is clear that ${}_{\FF} U$ is a simple (resp. projective) $k\C$-module if and only if $U$ is simple (resp. projective). Hence the group $\Aut(\C)$ acts on the set $\S(k\C)$. Moreover, if $b$ is a central idempotent of $k\C$, then $\Aut(\C)_b$ (the stabiliser of $b$ in $\Aut(\C)_b$) acts on the set $\S(k\C b)$. By \ref{defi: action of aut on weights of orbit category} (i) below, $\Aut(\C)$ also acts on the set of (isomorphism classes of) weights of $k\O_\C$. If $\C$ is an EI-category, then $\Aut(\C)_b$ acts on $\W(k\O_\C,b)$; see \ref{defi: action of aut on weights of orbit category} (ii).
	}
\end{notation}

The following conjecture is extracted from the inductive (blockwise) Alperin weight condition:

\begin{conjecture}[Navarro--Tiep {\cite[3.2]{NT}}; Sp\"{a}th {\cite[Definition 4.2 (ii)]{S}}]\label{conj:AW condition for finite groups}
	
	\begin{enumerate}[{\rm (i)}]
	
		\item  	For any finite group $G$, there exists a bijection $\S(kG)\to \W(k\O_G)$ commuting with the action of $\Aut(G)$.
		
		\item For any finite group $G$ and any central idempotent $b$ of $kG$, there exists a bijection $\S(kGb)\to \W(k\O_G,b)$ commuting with the action of $\Aut(G)_b$.
	\end{enumerate}

\end{conjecture}	

This leads to the following extension to finite categories:	

\begin{conjecture}\label{conj:AW Condition for category algebras}
\begin{enumerate}[{\rm (i)}]
	\item For any finite category $\C$, there exists a bijection $\S(k\C)\to \W(k\O_\C)$ commuting with the action of $\Aut(\C)$.
	\item For any finite EI-category $\C$ and any central idempotent $b$ of $k\C$, there exists a bijection $\S(k\C b)\to \W(k\O_\C,b)$ commuting with the action of $\Aut(\C)_b$.
\end{enumerate}	
\end{conjecture}

Combining Conjectures \ref{conj:AW Condition for category algebras} and \ref{conj:BGAWC for category algebras}, there is an even stronger conjecture:

\begin{conjecture}\label{conj:AW Condition plus Galois for category algebras}
Let $\C$ be a finite category (resp. EI-category) and $b$ the unit element (resp. a central idempotent) of $k\C$. Then there exists a bijection $\S(k\C b)\to \W(k\O_\C,b)$ commuting with the action of $(\Gamma\times\Aut(\C))_b$, the stabiliser of $b$ in $\Gamma\times\Aut(\C)$.
\end{conjecture}


As before, we reduce Conjectures \ref{conj:AW Condition for category algebras} and \ref{conj:AW Condition plus Galois for category algebras} to finite groups:

\begin{theorem}\label{theorem:main3}
	
	Let $\C$ be a finite category (resp. EI-category) and $b$ the unit element (resp. a central idempotent) of $k\C$.	
Then for any idempotent endomorphism $e$ in $\C$, $eb$ is a central idempotent in $k G_e$;  see Proposition \ref{prop:eb is in Ge}. Let $H$ be any subgroup of $\Gamma_b$. Assume that for any idempotent endomorphism $e$ in $\C$ there is an $(H\times\Aut(G_e))_{be}$-equivariant bijection $\S(k G_e eb)\to \W(k \O_{G_e},eb)$. Then there exists an $(H\times\Aut(\C))_b$-equivariant bijection $\S(k \C b)\to \W(k\O_{\C},b)$.
\end{theorem}

This will be proved in Section \ref{s:Automorphisms of categories}. 

\begin{corollary}
If Conjecture \ref{conj:AW Condition plus Galois for category algebras} holds for all finite groups, then it holds for all finite categories. In particular, Conjectures \ref{conj:AW condition for finite groups} and \ref{conj:AW Condition for category algebras} are equivalent.
\end{corollary}

	
	

\section{Twisted category algebras and their idempotent endomorphisms}\label{s2:Twisted category algebras and their idempotent endomorphisms}

Let $R$ be a commutative ring. For two $R$-algebras $A$ and $B$, an $R$-algebra isomorphism $\varphi:A\to B$ and an $A$-module $U$, we denote by ${}_\varphi U$ the $B$-module which is equal to $U$ as an $R$-module endowed with the structure homomorphism $B\xrightarrow{\varphi^{-1}} A\to \End_R(U)$. Unless otherwise specified, all modules (actions) in this paper are left modules (actions).  Let $\C$ be a finite category. The set of idempotent endomorphisms of objects in $\C$ is partially ordered, with partial order given by $e\leq f$ whenever $e$ and $f$ are idempotents endomorphisms of an object $X$ in $\C$ satisfying $e=e\circ f=f\circ e$. Two idempotent endomorphisms $e$ and $f$ of objects $X$ and $Y$, respectively, are called {\it isomorphic} if there are morphisms $s:X\to Y$ and $t:Y\to X$ satisfying $t\circ s=e$ and $s\circ t=f$. In this case, $s$ and $t$ can be chosen such that $s=f\circ s=s\circ e$ and $t=e\circ t=t\circ f$; indeed, using that $e$ and $f$ are idemptents, we have $e=t\circ s=e\circ t\circ s\circ t\circ s\circ e=(e\circ t\circ f)\circ(f\circ s\circ e)$, and a similar argument yields $f=(f\circ s\circ e)\circ(e\circ s\circ f)$.  Let $\alpha$ be a $2$-cocycle in $Z^2(\C;R^\times)$; that is, $\alpha$ is a map sending any two morphisms $s$ and $t$ in ${\rm Mor}(\C)$ for which $t\circ s$ is defined to a element $\alpha(t,s)$ in $R^\times$, such that for any three morphisms $s$, $t$ and $u$ for which the compositions $t\circ s$ and $u\circ t$ are defined, we have the {\it $2$-cocycle identity} $\alpha(u,t\circ s)\alpha(t,s)=\alpha(u\circ t,s)\alpha(u,t)$. The {\it twisted category algebra} $R_\alpha \C$ is the free $R$-module having the morphism set ${\rm Mor}(\C)$ as an $R$-basis, with an $R$-bilinear multiplication given by $ts=\alpha(t,s)t\circ s$ if $t\circ s$ is defined, and $ts=0$ otherwise. The $2$-cocycle identity is equivalent to the associativity of this multiplication. The isomorphism class of $R_\alpha \C$ depends only on the class of $\alpha$ in $H^2(\C;R^\times)$; see e.g. \cite[Theorem 1.4.7 (iii)]{Lin18a}. So if $\alpha$ represents the trivial class in $H^2(\C;R^\times)$ then $R_\alpha\C\cong R\C$, the usual category algebra of $\C$ over $R$. For any idempotent endomorphism $e$ of an object $X$ in $\C$, we denote by $G_e$ the group of all invertible elements in the monoid $e\circ \End_\C(X)\circ e$; following Linckelmann \cite{Lin14}, we call such a group $G_e$ {\it a maximal subgroup of $\C$}. The restriction of $\alpha$ to the group $G_e$ is abusively again denoted by $\alpha$. Note that the image in $R_\alpha\C$ of an idempotent endomorphism $e$ of an object in $\C$ is not necessarily an idempotent; more precisely, $e^2$ in $R_\alpha\C$ is equal to $\alpha(e,e)e$, and hence $\hat{e}=\alpha(e,e)^{-1}e$ is an idempotent in $R_\alpha\C$. However, we have $\hat{e}R_\alpha \C\hat{e}=eR_\alpha\C e$. Note that if $\C$ is an EI-category, then for any object $X$ in $\C$, ${\rm Id}_X$ is the unique idempotent in $\End_\C(X)$ and we have $G_{{\rm Id}_X}=\End_\C(X)$.

\begin{proposition}\label{prop: unit element}
Let $\C$ be a finite category and let $\alpha\in Z^2(\C;R^\times)$. The unit element of $R_\alpha\C$ is $\sum_{X\in \Ob(\C)}\alpha(\Id_X,\Id_X)^{-1}\Id_X$.
\end{proposition}

\begin{proof}
For any objects $X',Y$ in $\C$ and $s\in \Hom_\C(X',Y)$, we have 
$$s(\sum_{X\in \Ob(\C)}\alpha(\Id_X,\Id_X)^{-1}\Id_X)=\alpha(\Id_{X'},\Id_{X'})^{-1}\alpha(s,\Id_{X'})s=s,$$
where the second equality holds by the $2$-cocycle identity applied with $s$, $\Id_{X'}$, $\Id_{X'}$. Similarly, 
$$(\sum_{X\in \Ob(\C)}\alpha(\Id_X,\Id_X)^{-1}\Id_X)s=\alpha(\Id_{Y},\Id_{Y})^{-1}\alpha(\Id_{Y},s)s=s,$$
where the second equality holds by the $2$-cocycle identity applied with $\Id_Y$, $\Id_Y$, $s$.
\end{proof}

\begin{proposition}[{\cite[Propositions 5.2, 5.4]{LS}}]\label{prop:isomorphic idempotents}
Let $\C$ be a finite category and let $\alpha\in Z^2(\C;R^\times)$.	If $e$ and $f$ are isomorphic idempotents in $\C$, then there is an $R$-algebra isomorphism $\varphi: R_\alpha G_e \cong R_\alpha G_f$, which is uniquely determined up to an inner automorphism. More explicitly, the following hold:
	\begin{enumerate}[{\rm (i)}]
		\item Assume that $e$ and $f$ are idempotent endomorphisms of objects $X$ and $Y$, respectively, in $\C$. Assume that $s\in f\circ \Hom_\C(X,Y)\circ e$ and $t\in e\circ \Hom_\C(Y,X)\circ f$ satisfying $t\circ s=e$ and $s\circ t=f$. For $x\in G_e$, set $\beta(x)=\alpha(x,t)\alpha(s,x\circ t)\alpha(e,e)^{-1}\alpha(t,s)^{-1}$. Then the map $\varphi$ sending $x\in G_e$ to $\beta(x)(s\circ x\circ t)$ induces an $R$-algebra isomorphism $R_\alpha G_e\cong R_\alpha G_f$.  
		
		\item Assume that $s'\in f\circ \Hom_\C(X,Y)\circ e$ and $t'\in e\circ \Hom_\C(Y,X)\circ f$ satisfying $t'\circ s'=e$ and $s'\circ t'=f$. For any $y\in G_f$, set $\beta'(y)=\alpha(y,s')\alpha(t',y\circ s')\alpha(f,f)^{-1}\alpha(s',t')^{-1}$. Then the map sending $x\in G_e$ to $\beta(x)\beta'(s\circ x\circ t)(t'\circ s\circ x\circ t\circ s')$ induces an inner automorphism of the $R$-algebra $R_\alpha G_e$; it coincides with conjugation by $t'\circ s$ in $k_\alpha G_e$.
	\end{enumerate}
\end{proposition}

\begin{definition}\label{defi:isomorphism of pairs}
 {\rm Let $\C$ be a finite category and let $\alpha\in Z^2(\C;R^\times)$. Let $e$ be an idempotent endomorphism in $\C$. For any $R_\alpha\C$-module $U$, the $R$-space $eU=\hat{e}U$ is an $eR_\alpha\C e$-module (or equivalently, an $\hat{e}R_\alpha\C\hat{e}$-module), hence restricts to an $R_\alpha G_e$-module. Two pairs $(e,U)$ and $(f,V)$, consisting of idempotents $e\in \End_\C(X)$, $f\in \End_\C(Y)$, an $R_\alpha G_e$-module $U$ and an $R_\alpha G_f$ module $V$, are called {\it isomorphic} if the idempotents $e$ and $f$ are isomorphic and if the isomorphism classes of $U$ and $V$ correspond to each other through the induced isomorphism $\varphi:R_\alpha G_e\cong R_\alpha G_f$ in Proposition \ref{prop:isomorphic idempotents} (i). That is, ${}_\varphi U\cong V$ as $R_\alpha G_f$-modules. Since inner automorphisms of an $R$-algebra stabilise all isomorphism classes of modules, by Proposition \ref{prop:isomorphic idempotents} (ii) this property is independent of the choice of the isomorphism $\varphi:R_\alpha G_e\cong R_\alpha G_f$. }
\end{definition}

\begin{theorem}[{\cite[Theorem 1.2]{LS}}]\label{theo:bijection between simple moudles and pairs}
Let $\C$ be a finite category and let $\alpha\in Z^2(\C;R^\times)$. The map sending a simple $R_\alpha\C$-module $S$ to the pair $(e,eS)$, where $e$ is an idempotent endomorphism in $\C$ minimal with respect to $eS\neq 0$, induces a bijection $\Pi$ between $\S(R_\alpha \C)$ and the set of isomorphism classes of pairs $(e,T)$ consisting of an idempotent endomorphism $e$ in $\C$ and a simple $R_\alpha G_e$-module $T$.	
\end{theorem}

\begin{proposition}\label{prop: Galois action on pairs}
Let $\C$ be a finite category and let $\alpha\in Z^2(\C;\F_p^\times)\subseteq Z^2(\C;k^\times)$. There is a well-defined action of the Galois group $\Gamma=\Gal(k/\F_p)$ on the set of isomorphism classes of pairs $(e,T)$ consisting of an idempotent endomorphism $e$ in $\C$ and a simple $k_\alpha G_e$-module $T$ via ${}^\sigma(e,T)=(e,{}^\sigma T)$, where $\sigma\in \Gamma$. Sometimes we will also write ${}^\sigma(e,T)$ as ${}_\sigma(e,T)$. Moreover, $(e,T)$ is a weight if and only if $(e, {}^\sigma T)$ is a weight.
\end{proposition}

\begin{proof}
Let $\sigma\in \Gamma$. For the first statement we need to show that if $(e,T)$ is isomorphic to $(e',T')$, then $(e,{}^\sigma T)$ is isomorphic to $(e',{}^\sigma T')$. By Proposition \ref{prop:isomorphic idempotents} (i), there is a $k$-algebra isomorphism $\varphi:k_\alpha G_e\cong k_\alpha G_{e'}$. We need to show that ${}_\varphi({}^\sigma T)\cong {}^\sigma T'$ as $k_\alpha G_{e'}$-modules.  Since $\alpha$ is in $Z^2(\C;\F_p^\times)$, by the explicit construction of $\varphi$ in Proposition \ref{prop:isomorphic idempotents} (i), $\varphi$ is defined over $\F_p$; that is, there exists an $\F_p$-algebra isomorphism $\varphi_0:(\F_p)_\alpha G_e\cong (\F_p)_\alpha G_{e'}$ such that $\varphi={\rm Id}_k\otimes \varphi_0$. Since $(e,T)$ is isomorphic to $(e',T')$, there is an isomorphism $\psi: {}_\varphi T\cong T'$ of $k_\alpha G_{e'}$-modules. The map $\psi$ is also a $k_\alpha G_{e'}$-module isomorphism ${}^\sigma ({}_\varphi T) \cong {}^\sigma T'$. Now it suffices to show that ${}_\varphi({}^\sigma T)\cong{}^\sigma({}_\varphi T)$ as $k_\alpha G_{e'}$-modules. Indeed, the structure homomorphisms of the $k_\alpha G_{e'}$-modules ${}_\varphi({}^\sigma T)$ and ${}^\sigma({}_\varphi T)$ are, respectively, 
$$k_\alpha G_{e'} \xrightarrow{\sigma^{-1}} k_\alpha G_{e'} \xrightarrow{\varphi^{-1}} k_\alpha G_{e}\to \End_k(T),$$
and 
$$k_\alpha G_{e'} \xrightarrow{\varphi^{-1}} k_\alpha G_{e} \xrightarrow{\sigma^{-1}} k_\alpha G_{e}\to \End_k(T).$$
By definition (see Notation \ref{nation:Galois twist}), $\sigma^{-1}:k_\alpha G_e\to k_\alpha G_e$ and $\sigma^{-1}:k_\alpha G_{e'}\to k_\alpha G_{e'}$ can be written, respectively, as $\sigma^{-1}\otimes {\rm Id}_{(\F_p)_\alpha G_e}$ and $\sigma^{-1}\otimes {\rm Id}_{(\F_p)_\alpha G_{e'}}$. Since $\varphi^{-1}={\rm Id}_k\otimes \varphi_0^{-1}$, by the commutativity of the tensor product we have 
$$\varphi^{-1}\circ\sigma^{-1}=(\Id_k\otimes \varphi_0^{-1})\circ(\sigma^{-1}\otimes {\rm Id}_{(\F_p)_\alpha G_e})=\sigma^{-1}\otimes\varphi_0^{-1}=(\sigma^{-1}\otimes {\rm Id}_{(\F_p)_\alpha G_{e'}})\circ(\Id_k\otimes \varphi_0^{-1})=\sigma^{-1}\circ \varphi^{-1},$$
which proves the claim. As explained in Notation \ref{nation:Galois twist}, $T$ is projective if and only if ${}^\sigma T$ is projective, whence the second statement.
\end{proof}

\begin{proposition}\label{prop:pi commutes with gamma}
Keep the notation of Theorem \ref{theo:bijection between simple moudles and pairs}. Assume that $R=k$ and $\alpha\in Z^2(\C;\F_p^\times)$. Then the bijection $\Pi$ commutes with the action of the Galois group $\Gamma$.
\end{proposition}

\begin{proof}
Let $\sigma\in \Gamma$ and $S$ a simple $k_\alpha\C$-module. Denote by $[S]$ the isomorphism class of $S$. Let $e$ be an idempotent endomorphism in $\C$, minimal with respect to $eS\neq 0$. Then ${}^\sigma S$ is a simple $k_\alpha\C$-module, $e({}^\sigma S)$ is a $k_\alpha G_e$-module, and $e({}^\sigma S)\cong {}^\sigma(eS)$ as $k_\alpha G_e$-modules. Hence by the minimality of $e$, we see that $e$ is also minimal with respect to $e({}^\sigma S)\neq 0$. Therefore, we have 
$$\Pi([{}^\sigma S])=[(e,e({}^\sigma S))]=[(e,{}^\sigma(eS))]={}^\sigma[(e,eS)]={}^\sigma(\Pi([S])),$$
where the notation $[(e,S)]$ denotes the isomorphism class of $(e,S)$, and where the third equality holds by Proposition \ref{prop: Galois action on pairs}. 
\end{proof}

\begin{notation}\label{notation:lambda_ef}
	{\rm Let $\C$ be a finite category and $\alpha\in Z^2(\C;k^\times)$. For any two isomorphic idempotents $e$ and $f$ in $\C$, we denote by $\Lambda_{e,f}:\S(k_\alpha G_e)\to \S(k_\alpha G_f)$ the bijection sending $[T]\in \S(k_\alpha G_e)$ to $[{}_\varphi T]$, where $\varphi$ is the isomorphism $k_\alpha G_e\cong k_\alpha G_f$ as described in Proposition \ref{prop:isomorphic idempotents} (i). Since inner automorphisms of a $k$-algebra stabilise all isomorphism classes of modules, $\Lambda_{e,f}$ is independent of the choice of $\varphi$. It is clear that the inverse of $\Lambda_{e,f}$ is $\Lambda_{f,e}$. Moreover, one easily checks that for any three isomorphic idempotents $d$, $e$ and $f$ in $\C$, we have $\Lambda_{e,f}\circ \Lambda_{d,e}=\Lambda_{d,f}$. If $\alpha\in Z^2(\C;\F_p^\times)$, then $\Lambda_{e,f}$ commutes with the action of $\Gamma$. Indeed, by the explicit construction of the isomorphism $\varphi$ in Proposition \ref{prop:isomorphic idempotents} (i), $\varphi$ is defined over $\F_p$. For any $\sigma\in \Gamma$, by the proof of Proposition \ref{prop: Galois action on pairs}, we have ${}^\sigma({}_{\varphi} T)\cong {}_{\varphi}({}^\sigma T)$ as $k_\alpha G_f$-modules. }
\end{notation}

\begin{proposition}\label{propositon:isomorphic idempotents in T}
Let $\C$ be a finite category. Let $(X,P)$ and $(Y,Q)$ be objects in $\T_\C$. Let $(e,P,P)$ and $(f,Q,Q)$ be idempotent endomorphisms of $(X,P)$ and $(Y,Q)$, respectively. The following are equivalent:
\begin{enumerate}[{\rm (i)}]
	\item The morphisms $(e,P,P)$ and $(f,Q,Q)$ are isomorphic in $\T_\C$.
	\item There exist $s\in f\circ \Hom_\C(X,Y)\circ e$ and $t\in e\circ \Hom_\C(Y,X)\circ f$ satisfying $t\circ s=e$, $s\circ t=f$ and $s\circ (e\circ P)\circ t=f\circ Q$.
\end{enumerate}
\end{proposition}

\begin{proof}
Since $(e,P,P)$ and $(f,Q,Q)$ are respectively endomorphisms of $(X,P)$ and $(Y,Q)$, we have 
$$e=e\circ 1_P=1_P\circ e,~~~e\circ P\subseteq P\circ e,~~~f=f\circ 1_Q=1_Q\circ f~~~{\rm and}~~~f\circ Q\subseteq Q\circ f.$$ 
 From the inclusion $e\circ P\subseteq P\circ e$ we see that $e\circ P=e\circ P\circ e$ is a $p$-subgroup of $G_e$. Similarly, $f\circ Q=f\circ Q\circ f$ is a $p$-subgroup of $G_f$. 

Suppose (i) holds. By definition of being isomorphic, there exist morphisms $s'\in \Hom_\C(X,Y)$ and $t'\in \Hom_\C(Y,X)$ such that 
$$s'=s'\circ 1_P=1_Q\circ s',~~~t'=t'\circ 1_Q=1_P\circ t',~~~t'\circ s'=e,~~~s'\circ t'=f$$ 
$$s'\circ P\subseteq Q\circ s'~~~{\rm and}~~~t'\circ Q\subseteq P\circ t'.$$
Set $s=f\circ s'\circ e$ and $t=e\circ t'\circ f$. Then we have $t\circ s=e$ and $s\circ t=f$.
Consider the inclusion $s'\circ P\subseteq Q\circ s'$. Composing with $f$ and precomposing with $e$ yields $f\circ s'\circ P\circ e\subseteq f\circ Q\circ s'\circ e$. Using $e\circ P\circ e=e\circ P\subseteq P\circ e$ and $f\circ Q=f\circ Q\circ f$, we obtain 
$(f\circ s'\circ e) \circ (e\circ P)\subseteq (f\circ Q)\circ (f\circ s'\circ e)$, or equivalently, 
$$s\circ (e\circ P)\subseteq (f\circ Q)\circ s.$$ 
A similarly argument applied to the inclusion $t'\circ Q\subseteq P\circ t'$ shows that 
$$t\circ (f\circ Q)\subseteq (e\circ P)\circ t.$$
Using that $t\circ s=e$ and $s\circ t=f$, we easily see that the two inclusions above should be equalities. Thus (i) implies (ii).

Now suppose (ii) holds. Consider the objects $(X,e\circ P)$ and $(Y,f\circ Q)$ in $\T_\C$. The equalities $e\circ P=e\circ P\circ e$ and $f\circ Q=f\circ Q\circ f$ implies that $(e,e\circ P,e\circ P)$ is an idempotent endomorphism of $(X,e\circ P)$ and $(f,f\circ Q,f\circ Q)$ is an idempotent endomorphism of $(Y,f\circ Q)$. By assumption, we see that $s$ and $t$ define morphisms $(s,e\circ P,f\circ Q):(X,e\circ P)\to (Y,f\circ Q)$ and $(t,f\circ Q,e\circ P):(Y,f\circ Q)\to (X,e\circ P)$ in $\T_\C$.
Moreover, we have $(t,f\circ Q,e\circ P)\circ (s,e\circ P,f\circ Q)=(e,e\circ P,e\circ P)$ and $(s,e\circ P,f\circ Q)\circ (t,f\circ Q,e\circ P)=(f,f\circ Q,f\circ Q)$. Hence $(e,e\circ P,e\circ P)$ and $(f,f\circ Q,f\circ Q)$ are isomorphic idempotents in $\T_\C$. Then by \cite[Lemma 3.1 (iv)]{Lin14}, $(e,P,P)$ and $(f,Q,Q)$ are isomorphic in $\T_\C$. This shows that (i) is a consequence of (ii).
\end{proof}

\begin{proposition}\label{propositon:NGe(eP)/eP}
	Let $\C$ be a finite category and $\alpha\in Z^2(\O_\C;\F_p^\times)$. Let $(X,P)$ and $(Y,Q)$ be objects in $\T_\C$. Assume that $(e,P,P)$ and $(f,Q,Q)$ are isomorphic idempotent endomorphisms in $\T_\C$, or equivalently, there exist $s\in f\circ \Hom_\C(X,Y)\circ e$ and $t\in e\circ \Hom_\C(Y,X)\circ f$ satisfying $t\circ s=e$, $s\circ t=f$ and $s\circ (e\circ P)\circ t=f\circ Q$. The $k$-algebra isomorphism $\varphi:k_\alpha G_e\cong k_\alpha G_f$ described in Proposition \ref{prop:isomorphic idempotents} restricts to a $k$-algebra isomorphism $k_\alpha N_{G_e}(e\circ P)\cong k_\alpha N_{G_f}(f\circ Q)$, which in turn induces a $k$-algebra isomorphism $\bar{\varphi}:k_\alpha N_{G_e}(e\circ P)/(e\circ P)\cong k_\alpha N_{G_f}(f\circ Q)/(f\circ Q)$.
\end{proposition}

\begin{proof}
	The map sending $x\in G_e$ to $s\circ x\circ t$ is a group isomorphism $G_e\cong G_f$ and it restricts to group isomorphisms $e\circ P\cong f\circ Q$ and $N_{G_e}(e\circ P)\cong N_{G_f}(f\circ Q)$. Hence it induces a group isomorphism $N_{G_e}(e\circ P)/(e\circ P)\cong N_{G_f}(f\circ Q)/(f\circ Q)$. By the explicit construction of the isomorphism $\varphi$, we see that $\varphi$ restricts to a $k$-linear isomorphism $k_\alpha N_{G_e}(e\circ P)\cong k_\alpha N_{G_f}(f\circ Q)$. This is also a $k$-algebra isomorphism because $\varphi$ respects multiplication. Since $\alpha$ is taken from $Z^2(\O_\C;\F_p^\times)$, for any $x,y\in N_{G_e}(e\circ P)$, $\alpha(x,y)=1$ if one of $x$ or $y$ is in $e\circ P$. Similarly, for any $x',y'\in N_{G_f}(f\circ Q)$, $\alpha(x',y')=1$ if one of $x'$ or $y'$ is in $f\circ Q$. Hence the isomorphism $k_\alpha N_{G_e}(e\circ P)\cong k_\alpha N_{G_f}(f\circ Q)$ induces a $k$-algebra isomorphism $\bar{\varphi}:k_\alpha N_{G_e}(e\circ P)/(e\circ P)\cong k_\alpha N_{G_f}(f\circ Q)/(f\circ Q)$.
\end{proof}

\begin{notation}\label{notation: extension of category}
	{\rm Let $\C$ be a finite category, $R$ a finite subfield of $k$ and $\alpha\in Z^2(\C;R^\times)\subseteq Z^2(\C,k^\times)$. The {\it extension of $\C$ by $R^\times$ associated with $\alpha$} is the category $\hat{\C}$ with object set $\Ob(\C)$ and morphism set $\Mor(\hat{\C})=\Mor(\C)\times R^\times$, such that the composition in $\hat{\C}$ is defined by $(t,b)\circ(s,a)=(t\circ s,ba\alpha(t,s))$ for any two morphisms $s,~t\in \C$ for which $t\circ s$ is defined and any $a,~b\in R^\times$. There is a canonical functor $\hat{\C}\to \C$ which is the identity on objects and which sends a morphism $(s,a)$ in $\hat{C}$ to the morphism $s$ in $\C$. The isomorphism class of the extension $\hat{\C}$ of $\C$ by $R^\times$ depends only on the class of $\alpha$ in $H^2(\C;R^\times)$. It is clear that $\hat{\C}$ is finite. If $\C$ is an EI-category, then $\hat{\C}$ is an EI-category as well.} 
\end{notation}

\begin{proposition}\label{prop: isomorphism between khatC and kalphaC}
	Keep the notation of \ref{notation: extension of category}. The following hold:
	\begin{enumerate}[{\rm (i)}]
		\item The map sending a morphism $(s,a)$ in $\hat{\C}$ to $as\in k_\alpha \C$ induces a surjective $k$-algebra homomorphism $\pi:k\hat{\C}\to k_\alpha\C$.
		
		\item Let $e_{R}=\frac{1}{|R^\times|}\sum_{r\in R^\times}\sum_{X\in \Ob(\C)}r^{-1}(\Id_X,\alpha(\Id_X,\Id_X)^{-1}r)$. Then $e_R$ is a central idempotent in $k\hat{\C}$. The homomorphism $\pi$ in (i) restricts to a $k$-algebra isomorphism $k\hat{\C}e_R\cong k_\alpha\C$.
	\end{enumerate}
	
\end{proposition}

\begin{proof}
	It is straightforward to check the statement (i) and the fact that $e_R$ is a central idempotent in $k\hat{\C}$. Let $1$ be the unit element of $k_\alpha\C$; by Proposition \ref{prop: unit element} we have $1=\sum_{X\in \Ob(\C)}\alpha(\Id_X,\Id_X)^{-1}\Id_X$. Since $e_R$ is a central idempotent in $k\hat{\C}$, we have $k\hat{\C}=k\hat{\C}e_R\oplus k\hat{\C}(1-e_R)$. By the definition of $\pi$ we see that $\pi(e_R)=1$. Hence we obtain by restriction a surjective unitary $k$-algebra homomorphism $k\hat{\C}e_R\to k_\alpha \C$. In order to prove that this is an isomorphism, it suffices to show that $\{(s,1_R)e_R\mid s\in \Mor(\C)\}$ is a $k$-basis of $k\hat{\C}e_R$. Indeed, for any objects $X,Y$ in $\C$, $s\in \Hom_\C(X,Y)$ and $a\in R^\times$, an easy calculation shows that 
	$$(s,a)e_R=a(s,1_R)e_R.$$
	Hence $\{(s,1_R)e_R\mid s\in \Mor(\C)\}$ is a $k$-basis of $k\hat{\C}e_R$.
\end{proof}	

\begin{lemma}[{\cite[Lemma 4.3 (i), (iii)]{LS}}]\label{lemma:correspondence of idempotents in extension} 	Keep the notation of \ref{notation: extension of category}. The following hold:
	\begin{enumerate}[{\rm (i)}]
		\item The canonical functor $\hat{\C}\to \C$ induces a bijection between the set of idempotents in $\hat{\C}$ and the set of idempotents in $\C$. The inverse of this bijection sends an idempotent $e$ in $\C$ to the idempotent $(e,\alpha(e,e)^{-1})$ in $\hat{\C}$.
		
		\item For any two idempotents $e$ and $f$ in $\C$, $e$ is isomorphic to $f$ if and only if $(e,\alpha(e,e)^{-1})$ is isomorphic to $(f,\alpha(f,f)^{-1})$ in $\hat{\C}$. More precisely, if $s\in f\circ \Hom_\C(X,Y)\circ e$ and $t\in e\circ \Hom_\C(Y,X)\circ f$ satisfies $t\circ s=e$ and $s\circ t=f$, then setting $c=\alpha(s,t)^{-1}\alpha(f,f)^{-1}$, we have $(t,c)\circ (s,1_R)=(e,\alpha(e,e)^{-1})$ and $(s,1_R)\circ (t,c)=(f,\alpha(f,f)^{-1})$.
	\end{enumerate}

\end{lemma}

\begin{proposition}\label{prop:commutative diagram for isomorphisms between idempotents}
	Keep the notation of \ref{notation: extension of category}. Let $X$ be an object in $\C$ and $e\in \End_\C(X)$ an idempotent. Then by Lemma \ref{lemma:correspondence of idempotents in extension}, $(e,\alpha(e,e)^{-1})$ is an idempotent in $\End_{\hat{\C}}(X)$. One checks that the invertible elements in the monoid $(e,\alpha(e,e)^{-1})\End_{\hat{\C}}(X)(e,\alpha(e,e)^{-1})$ is exactly the group $G_e\times R^\times$ with the multiplication described in Notation \ref{notation: extension of category}. We denote this group by $\hat{G}_e$. Then the following hold:
	\begin{enumerate}[{\rm (i)}]
		\item The $k$-algebra homomorphism (resp. isomorphism) $\pi:k\hat{\C}\to k_\alpha\C$  (resp. $\pi:k\hat{\C}e_R\to k_\alpha\C$) in Proposition \ref{prop: isomorphism between khatC and kalphaC} restricts to a $k$-algebra homomorphism (resp. isomorphism) $\pi_e: k\hat{G}_e\cong k_\alpha G_e$ (resp. $\pi_e:e_Rk\hat{G}_ee_R =k\hat{G}_ee_R\cong k_\alpha G_e$).
		\item Let $Y\in \Ob(\C)$ and $f\in \End_\C(Y)$ be an idempotent isomorphic to $e$, then by Lemma \ref{lemma:correspondence of idempotents in extension}, the idempotent $(f,\alpha(f,f)^{-1})$ is isomorphic to $(e,\alpha(e,e)^{-1})$. Let $s$, $t$ and $c$ be chosen as in Lemma \ref{lemma:correspondence of idempotents in extension} (ii). Let $\hat{\varphi}$ be the group isomorphism $\hat{G}_e\cong \hat{G}_f$ sending $\hat{x}\in \hat{G}_e$ to $(s,1_R)\circ \hat{x}\circ (t,c)$. Then $\hat{\varphi}$ extends to a $k$-algebra isomorphism $k\hat{G}_e\cong k\hat{G}_f$, which we still denote by $\hat{\varphi}$. Let $\varphi$ be the $k$-algebra isomorphism $k_\alpha G_e\cong k_\alpha G_f$ described in Proposition \ref{prop:isomorphic idempotents} (i). The diagram 
		$$\xymatrix{k\hat{G}_ee_R\ar[rr]^{\hat{\varphi}} \ar@{_{(}->}[d] &   &   k\hat{G}_fe_R \ar@{^{(}->}[d] \\
			k\hat{G}_e\ar[rr]^{\hat{\varphi}} \ar[d]_{\pi_e} &   &   k\hat{G}_f \ar[d]^{\pi_f}   \\
			k_\alpha G_e \ar[rr]^{\varphi}  &    &  k_\alpha G_f
		}$$
		is commutative.
		
	\end{enumerate}
\end{proposition}

\begin{proof}
	Statement (i) is a straightforward verification by using the explicit construction of the map $\pi$ described in Proposition \ref{prop: isomorphism between khatC and kalphaC} (i). For the statement (ii), let $(x,\lambda)\in \hat{G}_e$, where $x\in G_e$ and $\lambda\in R^\times$. Then 
	\begin{align*}
		\begin{split}
			\pi_f\circ\hat{\varphi}((x,\lambda)) &=\pi_f((s,1_R)\circ (x,\lambda) \circ (t,c)) \\
			&=\lambda c\alpha(s,x)\alpha(s\circ x,t) s\circ x\circ t\\
			&=\lambda\alpha(s,t)^{-1}\alpha(f,f)^{-1}\alpha(s,x)\alpha(s\circ x,t) s\circ x\circ t
		\end{split}
	\end{align*}
	and 
	\begin{align*}
		\begin{split}
			\varphi\circ \pi_e ((x,\lambda)) &=\varphi(\lambda x)=\lambda\beta(x)s\circ x\circ f\\
			&=\lambda\alpha(x,t)\alpha(s,x\circ t)\alpha(e,e)^{-1}\alpha(t,s)^{-1} s\circ x\circ f
		\end{split}
	\end{align*}
	The 2-cocycle identity applied with $s$, $x$, $t$ yields
	$$\alpha(s,x)\alpha(s\circ x,t)=\alpha(s,x\circ t)\alpha(x,t).$$
	By \cite[Lemma 5.3]{LS} we have 
	$$\alpha(t,s)\alpha(e,e)=\alpha(s,t)\alpha(f,f).$$
	Using the above two equalities, we obtain 
	$\pi_f\circ\hat{\varphi}((x,\lambda))=\varphi\circ \pi_e ((x,\lambda))$, whence (ii). 
\end{proof}

\section{Weights of $k_\alpha\O_\C$ and proof of Theorem \ref{theorem:main}}\label{s3:weights of O and proof of main}

\begin{definition}[{\cite[1.4]{LS}}]\label{defi:weight algebra}
{\rm Let $\C$ be a finite category and $\alpha\in Z^2(\C;k^\times)$. A {\it weight} of $k_\alpha\C$ is a pair $(e,T)$ consisting of an idempotent endomorphism $e$ of an object $X$ in $\C$ and a projective simple $k_\alpha G_e$-module $T$. 
	
}
\end{definition}	

In the rest of this section, let $\C$ be a finite category with $p$-transporter category $\T_\C$ and associated $p$-orbit category $\O_\C$. The canonical functor $\T_\C\to \O_\C$ is the identity on objects, and surjective on morphisms between any pair of objects in $\T_\C$. For any object $(X,P)$ in $\T_\C$, the kernel of the canonical moniod homomorphism $\End_{\T_\C}((X,P))\to\End_{\O_\C}((X,P))$ can be identified with $P$. For any two objects $(X,P)$ and $(Y,Q)$ in $\T_\C$, the canonical map $$\Hom_{\T_\C}((X,P),(Y,Q))\to \Hom_{\O_\C}((X,P),(Y,Q))$$ induces a bijection between the sets $Q\backslash \Hom_{\T_\C}((X,P),(Y,Q))$ and $\Hom_{\O_\C}((X,P),(Y,Q))$.

\begin{lemma}[{\cite[Lemmas 3.2, 3.3, 3.4]{Lin14}}]\label{lemma:maximal subgroups of orbit category}
Let $(X,P)$ and $(Y,Q)$ be objects in $\T_\C$. Identify morphisms between $(X,P)$ and $(Y,Q)$ in $\T_\C$ (resp. $\O_\C$) with their canonical images in $\Hom_\C(X,Y)$ (resp. $Q\backslash \Hom_\C(X,Y)/P$). Let $e\in \End_{\T_\C}((X,P))$ and $f\in \End_{\T_\C}((Y,Q))$ be idempotent endomorphisms. Denote by $\bar{e}=P\circ e \circ P=P\circ e$ the canonical image of $e$ in $\End_{\O_\C}((X,P))$, and by $G_e$ the group of invertible elements in the monoid $e\circ \End_\C(X)\circ e$.  The following hold:
\begin{enumerate}[{\rm (i)}]
	\item For any idempotent $d'\in\End_{\O_\C}((X,P))$, there is an idempotent $d\in \End_{\T_\C}((X,P))$ such that $d'=P\circ d \circ P$.
	\item  The idempotents $e$ and $f$ are isomorphic in $\T_\C$ if and only if $\bar{e}$ and $\bar{f}$ are isomorphic idempotents in $\O_\C$.
	\item The group of invertible elements in the monoid $e\circ \End_{\T_\C}((X,P))\circ e$ is equal to $N_{G_e}(e\circ P)$.
	\item The group of invertible elements in the monoid $\bar{e}\circ \End_{\O_\C}((X,P))\circ \bar{e}$ is equal to $N_{G_e}(e\circ P)/(e\circ P)$. 
\end{enumerate}
\end{lemma}

\begin{remark} 
	{\rm By Lemma \ref{lemma:maximal subgroups of orbit category}, we easily see that:
		
	 (i) The isomorphism $\varphi|_{k_\alpha N_{G_e}(e\circ P)}:k_\alpha N_{G_e}(e\circ P)\cong k_\alpha N_{G_f}(f\circ Q)$ in Proposition \ref{propositon:NGe(eP)/eP} coincides the isomorphism obtained by applying Proposition \ref{prop:isomorphic idempotents} (i) to $\T_\C$, $(e,P,P)$, $(f,Q,Q)$, $(s,P,Q)$, $(t,Q,P)$ instead of $\C$, $e$, $f$, $s$, $t$, respectively.
		
	(ii)	The isomorphism $\bar{\varphi}: k_\alpha N_{G_e}(e\circ P)/(e\circ P)\cong k_\alpha N_{G_f}(f\circ Q)/(f\circ Q)$ in Proposition \ref{propositon:NGe(eP)/eP} coincides the isomorphism obtained by applying Proposition \ref{prop:isomorphic idempotents} (i) to $\O_\C$, $\overline{(e,P,P)}$, $\overline{(f,Q,Q)}$, $\overline{(s,P,Q)}$, $\overline{(t,Q,P)}$ instead of $\C$, $e$, $f$, $s$, $t$, respectively. Here the overline notation denotes taking the canonical image in $\O_\C$.
		
	}	
\end{remark}

\begin{remark}\label{lemma:weights of orbit category}
{\rm Let $\alpha\in Z^2(\O_\C;k^\times)$. 
	
(i) Combining Definition \ref{defi:weight algebra} and Lemma \ref{lemma:maximal subgroups of orbit category}, we see that: a weight of $k_\alpha\O_\C$ is exactly a pair $(\bar{e},T)$, where $\bar{e}=P\circ e\circ P$ for some idempotent endomorphism $e$ of some object $X$ in $\C$, $P$ is a not necessarily unitary $p$-subgroup of the moniod $\End_\C(X)$, and $T$ is a projective simple $k_\alpha N_{G_e}(e\circ P)/(e\circ P)$-module. We identify the weight $(\bar{e},T)$, the quadruple $(e,P,P,T)$ and the pair $((e,P,P),T)$. Note that $(e,P,P)$ is an idempotent endomorphism of $(X,P)$ in the category $\T_\C$. Consider $G_e$ as a category with one object $X$, then $(e\circ P,T)=(P\circ e\circ P,T)$ is a weight of $k_\alpha\O_{G_e}$. 
	
(ii) Conversely, let $X$ be an object in $\C$ and $e\in \End_\C(X)$ an idempotent. A weight of $k_\alpha\O_{G_e}$ is of the form $(P, T)$, where $P$ is a $p$-subgroup of $G_e$, and $T$ is a projective simple $k_\alpha N_{G_e}(P)/P$-module. One easily sees that the pair $(P,T)=(P\circ e\circ P,T)$ is also a weight of $k_\alpha\O_\C$.

(iii)  Let $X$ and $Y$ be objects in $\C$. Let $e\in \End_\C(X)$ and $f\in \End_\C(Y)$ be a pair of isomorphic idempotents. Let $(P,T)$ be a weight of $k_\alpha \O_{G_e}$ and $(Q,T')$ a weight of $k_\alpha \O_{G_f}$. Denote by $[(P,T)]\in\W(k_\alpha\O_{G_e})$ (resp. $[(Q,T')]\in\W(k_\alpha\O_{G_f})$) the isomorphism class of $(P,T)$ (resp. $(Q,T')$). We say that $[(P,T)]$ and $[(Q,T')]$ are {\it isomorphic} if the pairs $((e,P,P),T)$ and $((f,Q,Q),T')$ are isomorphic in the sense of Definition \ref{defi:isomorphism of pairs}. In fact, this is equivalent to that as weights of $k_\alpha\O_\C$, $(P,T)$ and $(Q,T')$ are in the same isomorphism class.

(iv) Keep the notation in (iii) and assume that $[(P,T)]$ and $[(Q,T')]$ are isomorphic. Then by Proposition \ref{propositon:isomorphic idempotents in T}, there exist $s\in f\circ \Hom_\C(X,Y)\circ e$ and $t\in e\circ \Hom_\C(Y,X)\circ f$ such that $t\circ s=e$, $s\circ t=f$ and $Q=s\circ P\circ t$. Let $\varphi$ be the $k$-algebra isomorphism $\varphi:k_\alpha G_e\cong k_\alpha G_f$ described in Proposition \ref{prop:isomorphic idempotents}. By Proposition \ref{propositon:NGe(eP)/eP}, $\varphi$ induces a $k$-algebra isomorphism $\bar{\varphi}:k_\alpha N_{G_e}(P)/P\cong k_\alpha N_{G_f}(Q)/Q$. The pairs $((e,P,P),T)$ and $((f,Q,Q),T')$ being isomorphic implies that the isomorphism classes of $T$ and $T'$ correspond to each other via $\bar{\varphi}$ (that is, $T'\cong {}_{\bar{\varphi}} T$ as $k_\alpha N_{G_f}(Q)/Q$-modules). 
}
\end{remark}

\begin{lemma}[{\cite[Lemma 3.5]{Lin14}}]\label{lemma: L14 Lemma 3.5}
Let $\E$ be a set of representatives of the isomorphism classes of idempotent endomorphisms in $\C$. For any $e\in\E$, denote by $X_e$ the object in $\C$ of which $e$ is an idempotent endomorphism, by $G_e$ the subgroup of invertible elements of the monoid $e\circ \End_\C(X_e)\circ e$ and by $\X_e$ a set of representatives of the $G_e$-conjugacy classes of $p$-subgroups of $G_e$. Then the following hold:
\begin{enumerate}[{\rm (i)}]
	\item The set $\{(X_e,P)\mid e\in \E, ~P\in\X_e\}$ is a set of representatives of the isomorphism classes of objects in $\T_\C$.
	\item  The set $\{(e,P,P)\mid e\in \E, ~P\in\X_e\}$ is a set of representatives of the isomorphism classes of idempotent endomorphisms in $\T_\C$.
\end{enumerate}	
Combining (ii) and Lemma \ref{lemma:maximal subgroups of orbit category} (i), (ii) we also have: 
\begin{enumerate}[{\rm (iii)}]
\item The set $\{(\bar{e},P,P)\mid e\in \E, ~P\in\X_e\}$, where $\bar{e}=P\circ e\circ P$, is a set of representatives of the isomorphism classes of idempotent endomorphisms in $\O_\C$.
\end{enumerate} 
\end{lemma}

\begin{proposition}\label{proposition:bijection between isomorphic weights commutes with Galois}
	Let $\alpha\in Z^2(\O_\C;k^\times)$. Let $X$ and $Y$ be objects in $\C$. Let $e\in \End_\C(X)$ and $f\in \End_\C(Y)$ be a pair of isomorphic idempotents. 
	\begin{enumerate}[{\rm (i)}]
		\item For any weight $(P,T)$ of $k_\alpha \O_{G_e}$, there is a unique isomorphism class $[(Q,T')]$ of weight of $k_\alpha\O_{G_f}$ isomorphic to $[(P,T)]$. 
		
		\item Assume that $\alpha\in Z^2(\O_\C;\F_p^\times)$. The correspondence $[(P,T)]\mapsto [(Q,T')]$ defines a bijection $\Omega_{e,f}:\W(k_\alpha\O_{G_e})\to \W(k_\alpha\O_{G_f})$ which is  commuting with the action of the Galois group $\Gamma$. In other words, for any $\sigma\in \Gamma$, $[(P,T)]\cong [(Q,T')]$ if and only if $[(P,{}^\sigma T)]\cong [(Q,{}^\sigma T')]$.
		\item The inverse of the bijection $\Omega_{e,f}$ is $\Omega_{f,e}$.
		\item For any three isomorphic idempotents $d$, $e$ and $f$ in $\C$, we have $\Omega_{e,f}\circ \Omega_{d,e}=\Omega_{d,f}$.
	\end{enumerate}
	
\begin{proof}
By Lemma \ref{lemma: L14 Lemma 3.5} (ii), up to $G_f$-conjugation there is a unique $p$-subgroup $Q$ of $G_f$ such  that $(f,Q,Q)$ is isomorphic to $(e,P,P)$ in $\T_\C$. By Proposition \ref{propositon:isomorphic idempotents in T}, there exist $s\in f\circ \Hom_\C(X,Y)\circ e$ and $t\in e\circ \Hom_\C(Y,X)\circ f$ such that $t\circ s=e$, $s\circ t=f$ and $Q=s\circ P\circ t$. Let $\varphi$ be the $k$-algebra isomorphism $\varphi:k_\alpha G_e\cong k_\alpha G_f$ described in Proposition \ref{prop:isomorphic idempotents}. By Proposition \ref{propositon:NGe(eP)/eP}, $\varphi$ induces a $k$-algebra isomorphism $\bar{\varphi}:k_\alpha N_{G_e}(P)/P\cong k_\alpha N_{G_f}(Q)/Q$. Then there is a unique isomorphism class of simple $k_\alpha N_{G_f}(Q)$-module, say $[T']$, such that $[T]$ and $[T']$ correspond to each other via $\bar{\varphi}$. Hence $[(Q,T')]$ is the unique isomorphism class of weight of $k_\alpha\O_{G_f}$ isomorphic to $[(P,T)]$, proving statement (i). For statement (ii), note that ${}_{\bar{\varphi}}T\cong T'$ if and only if ${}^\sigma({}_{\bar{\varphi}}T)\cong {}^\sigma T'$. Since $\alpha$ is in $Z^2(\C;\F_p^\times)$, by the explicit construction of the isomorphism $\varphi$ in Proposition \ref{prop:isomorphic idempotents} (i), $\bar{\varphi}$ is defined over $\F_p$. Using a similar argument as in the proof of Proposition \ref{prop: Galois action on pairs}, we have ${}^\sigma({}_{\bar{\varphi}} T)\cong {}_{\bar{\varphi}}({}^\sigma T)$ as $k_\alpha N_{G_f}(Q)/Q$-modules. Therefore, ${}_{\bar{\varphi}}T\cong T'$ if and only if ${}_{\bar{\varphi}}({}^\sigma T)\cong {}^\sigma T'$, whence (ii). Statements (iii) and (iv) are trivial.
\end{proof}
	
\end{proposition}

\begin{proof}[Proof of Theorem \ref{theorem:main}]
We will use the notation of Lemma \ref{lemma: L14 Lemma 3.5}. By Theorem \ref{theo:bijection between simple moudles and pairs}, there is a bijection
$$\Pi: \S(k_\alpha \C)\to \bigsqcup_{e\in \E} \S(k_\alpha G_e),$$
where the symbol $\bigsqcup$ denotes the disjoint union.
By Proposition \ref{prop:pi commutes with gamma}, $\Pi$ is commuting with the action of $\Gamma$. By assumption, the GAWC holds for $k_\alpha G_e$, that is, for any $e\in \E$ there is a $\Gamma$-equivariant bijection 
$$\S(k_\alpha G_e)\to \bigsqcup_{P\in \X_e}\mathcal{Z}(k_\alpha N_{G_e}(P)/P),$$
where $\mathcal{Z}(k_\alpha N_{G_e}(P)/P)$ is the set of isomorphism classes of projective simple $k_\alpha N_{G_e}(P)/P$-modules. Hence there is a bijection
$$\S(k_\alpha \C)\to \bigsqcup_{e\in\E}\bigsqcup_{P\in \X_e} \mathcal{Z}(k_\alpha N_{G_e}(P)/P)$$
commuting with the action of $\Gamma$. It remains to show that the right side corresponds to a set of representatives of the isomorphism classes of weights of $k_\alpha \O_\C$. In this double union, $e$ runs over $\E$ and $P$ over $\X_e$. By Lemma \ref{lemma: L14 Lemma 3.5} (ii), this implies that the triples $(e,P,P)$ runs over a set of representatives of the isomorphism classes of idempotent endomorphisms in $\T_\C$. By Lemma \ref{lemma:maximal subgroups of orbit category} (ii), the images of the triples in the morphism sets of $\O_\C$ runs over a set of representatives of the isomorphism classes of idempotent endomorphisms in $\O_\C$. By Lemma \ref{lemma:maximal subgroups of orbit category} (iv), the maximal subgroup determined by the image of any such $(e,P,P)$ in $\O_\C$ is $N_{G_e}(P)/P$, and hence when $[S]$ runs over $\mathcal{Z}(k_\alpha N_{G_e}(P)/P)$, the quadruple $(e,P,P,S)$ runs over a set of representatives of the isomorphism classes of weights of $k_\alpha\O_\C$ associated with the image of $(e,P,P)$ in $\O_\C$. Therefore, the double union corresponds to a set of representatives the isomorphism classes of weights of $k_\alpha \O_\C$.
\end{proof}

\section{The Brauer construction applied to twisted group algebras}\label{s:Brauer construction}

Let $G$ be a finite group. A {\it $G$-algebra over $k$} is a $k$-algebra $A$ endowed with an action of $G$ by $k$-algebra automorphisms, denoted $a\mapsto {}^ga$, where $a\in A$ and $g\in G$. For any $p$-subgroup $P$ of $G$, we denote by $A^P$ the $N_G(P)$-subalgebra of $P$-fixed points of $A$.  For any two $p$-subgroups $Q\leq P$ of $G$, the {\it relative trace map} ${\rm Tr}_Q^P:A^Q\to A^P$ is defined by ${\rm Tr}_Q^P(a)=\sum_{x\in [P/Q]}{}^xa$, where $[P/Q]$ denotes a set of representatives of the left cosets of $Q$ in $P$. We denote by $A(P)$ the {\it $P$-Brauer quotient} of $A$, i.e., the $N_G(P)$-algebra
\[A^P/\sum_{Q<P}{\rm Tr}_Q^P(A^Q).\]
We denote by ${\rm br}_P^A:A^P\to A(P)$ the canonical map, which is called the {\it $P$-Brauer homomorphism}. 
If $f:A\to B$ is a homomorphism of $G$-algebras, then $f$ restricts to a homomorphism of $N_G(P)$-algebras $f^P:A^P\to B^P$, which in turn induces a homomorphism of $N_G(P)$-algebras
$f(P): A(P)\to B(P)$
sending $\br_P^A(a)$ to $\br_P^B(f(a))$ for any $a\in A^P$. In other words, we have $f(P)\circ \br_P^A=\br_P^B\circ f^P$. In this way, the $P$-Brauer construction defines a functor (called the {\it $P$-Brauer functor}) from the category of $G$-algebras to the category of $N_G(P)$-algebras. 

\begin{lemma}[cf. e.g. {\cite[Exercise 11.4 or Proposition 27.6 (a)]{Thevenaz}}]\label{lemma:basis}
	Let $G$ be a finite group and $P$ a $p$-subgroup of $G$. Let $A$ be a $G$-algebra over $k$. If $A$ has a $P$-stable $k$-basis $X$, then $\{\br_P^A(a)\mid a\in X^P\}$ is a $k$-basis of $A(P)$, where $X^P:=\{x\in X\mid {}^ux=x,~\forall~u\in P\}$.
\end{lemma}

\begin{void}\label{void:central extension}
	{\rm Let $G$ be a finite group, $P$ a $p$-subgroup of $G$ and $\alpha\in Z^2(G;k^\times)$. Then both $kG$ and $k_\alpha G$ are $G$-algebras via conjugation action of $G$. Since $kG$ has an obvious $P$-stable $k$-basis $G$, then by Lemma \ref{lemma:basis}, we have $(kG)(P)\cong kC_G(P)$. But it is not very obvious that $k_\alpha G$ has a $P$-stable $k$-basis. To obtain $(k_\alpha G)(P)\cong k_\alpha C_G(P)$, we can not directly use Lemma \ref{lemma:basis}. Since $G$ is finite and since any finite subset of $k=\bF_p$ is contained in a finite subfield of $k$, we may assume that $\alpha\in Z^2(G;R^\times)$ for a finite subfield $R$ of $k$. Let $\hat{G}$ 
		be the extension of $G$ by $R^\times$ associated with $\alpha$; see Notation \ref{notation: extension of category}. Recall from Notation \ref{notation: extension of category} and Proposition \ref{prop: isomorphism between khatC and kalphaC} that $\hat{G}$ has the following properties: 
		\begin{enumerate}[{\rm (i)}]
			\item $\hat{G}=\{(g,r)\mid g\in G,~r\in R^\times\}$. Denote by $\tau$ the canonical surjection $\hat{G}\to G$ sending $(g,r)$ to $g$ for any $g\in G$ and $r\in R^\times$. Then $\ker(\tau)=\{(1,r)\mid r\in R^\times\}\cong R^\times$ is a central subgroup of $\hat{G}$ of order not divisible by $p$.
			\item Write $e_R:=\frac{1}{|R^\times|}\sum_{r\in R^\times}r^{-1} (1,\alpha(1,1)^{-1}r)$. Then $e_R$ is an idempotent in $Z(k\hat{G})$ and we have a $k$-algebra isomorphism $\pi:k\hat{G}e_R\to k_\alpha G$ sending $(g,r)\in \hat{G}$ to $rg\in k_\alpha G$.
		\end{enumerate}
		Since $P$ is a $p$-subgroup of $G$ and since $p\nmid |\ker(\tau)|$, there is a $p$-subgroup $\hat{P}$ of $\hat{G}$ such that $\hat{P}\cong P$ via $\tau$. Denote by $\hat{u}$ the preimage of $u\in P$ in $\hat{P}$. Then $\hat{P}=\{\hat{u}\mid u\in P\}$.	
	}
\end{void}

\begin{proposition}\label{prop:isomorphic as G-algebras}
	Keep the notation of \ref{void:central extension}.  For any $t\in k\hat{G}$ and $g\in G$, set ${}^gt=(g,1)t(g,1)^{-1}$; clearly ${}^gt=(g,r)t(g,r)^{-1}$ for any $r\in R^\times$. This defines an action of $G$ on $k\hat{G}$, and the isomorphism $\pi:k\hat{G}e_R\cong k_\alpha G$ in \ref{void:central extension} is an isomorphism of $G$-algebras. 
\end{proposition}

\begin{proof}
	This can be checked straightforward by definition.
\end{proof}

\begin{proposition}\label{prop:restrictions of twisted group algebras}
	Keep the notation of \ref{void:central extension}. The following hold:
	\begin{enumerate}[{\rm (i)}]
		\item $\tau^{-1}(C_G(P))=C_{\hat{G}}(\hat{P})$ and $\tau^{-1}(N_G(P))=N_{\hat{G}}(\hat{P})$.
		\item The $G$-algebra isomorphism $\pi:k\hat{G}e_R\cong k_\alpha G$ restricts to isomorphisms of $N_G(P)$-algebras $$kC_{\hat{G}}(\hat{P})e_R\cong k_\alpha C_G(P)$$ and $$kN_{\hat{G}}(\hat{P})e_R\cong k_\alpha N_G(P),$$ where the restrictions of $\alpha$ to subgroups of $G$ are abusively again denoted by $\alpha$.
		\item The $N_G(P)$-algebra isomorphism $kN_{\hat{G}}(\hat{P})e_R\cong k_\alpha N_G(P)$ induces an $N_G(P)$-algebra isomorphism
		$$kN_{\hat{G}}(\hat{P})/\hat{P}\bar{e}_R\cong k_\alpha N_G(P)/P,$$
		where $\bar{e}_R$ is the canonical image of $e_R$ in $N_{\hat{G}}(\hat{P})/\hat{P}$.
		\item If $\alpha\in Z^2(\C;\F_p^\times)$, then all isomorphisms in (ii) and (iii) are defined over $\F_p$.
	\end{enumerate}
\end{proposition}

\begin{proof}
	Clearly we have $C_{\hat{G}}(\hat{P})\subseteq \tau^{-1}(C_G(P))$ and $N_{\hat{G}}(\hat{P})\subseteq \tau^{-1}(N_G(P))$. For any $a\in \tau^{-1}(C_G(P))$, since $\tau(a)\in C_G(P)$, we have $a\hat{u}a^{-1}\hat{u}^{-1}\in \ker(\tau)$ for all $u\in P$. It follows that $a\hat{u}a^{-1}=\hat{u}z$ for some $z\in \ker(\tau)$. Since $a\hat{u}a^{-1}$ is a $p$-element of $\hat{G}$, $\hat{u}z$ should also be a $p$-element of $\hat{G}$. This forces $z=1$ and hence $a\in C_{\hat{G}}(\hat{P})$.  Let $a\in \tau^{-1}(N_G(P))$.  Since $\tau(a)\in N_G(P)$, for any $u\in P$ there exists $v\in P$ such that $\tau(a)u\tau(a)^{-1}=v$. Equivalently, we have $\tau(a)\tau(\hat{u})\tau(a^{-1})=\tau(\hat{v})$. This implies that $a\hat{u}a^{-1}=\hat{v}z$ for some $z\in \ker(\tau)$. Since $a\hat{u}a^{-1}$ is a $p$-element of $\hat{G}$, $\hat{v}z$ should also be a $p$-element of $\hat{G}$. This again forces $z=1$ and hence $a\in N_{\hat{G}}(\hat{P})$, completing the proof of (i). Statements (ii), (iii) and (iv) follows from (i) and the explicit construction of $\pi$ in \ref{void:central extension} (iii). 
\end{proof}

Let $A$ be a finite-dimensional $k$-algebra. By a {\it block} of $A$, we mean a primitive central idempotent of $A$. The $k$-algebra $Ab=bAb$ is called a block algebra of $A$. For any indecomposable $A$-module $U$, there is a unique block $b$ of $A$ such that $bU\neq 0$, and hence $U$ is an $Ab$-module.  

\begin{proposition}\label{prop:Brauer map and central idempotents}
	Let $G$ be a finite group, $P$ a $p$-subgroup of $G$ and $\alpha\in Z^2(G;k^\times)$. The following hold: 
	\begin{enumerate}[{\rm (i)}]
		\item $k_\alpha C_G(P)\subseteq (k_\alpha G)^P$. 
		\item The composition $k_\alpha C_G(P)\hookrightarrow (k_\alpha G)^P\xrightarrow{\br_P^{k_\alpha G}} (k_\alpha G)(P)$ is an isomorphism of $N_G(P)$-algebras. 
		\item We abusively use the same symbol $\br_P^{k_\alpha G}$ to denote the following composition of $N_G(P)$-algebra homomorphisms 
		$$(k_\alpha G)^P\to (k_\alpha G)(P)\cong k_\alpha C_G(P)\hookrightarrow k_\alpha N_G(P),$$ where the isomorphism $(k_\alpha G)(P)\cong k_\alpha C_G(P)$ is the inverse of the isomorphism obtained in (ii). Then $\br_P^{k_\alpha G}$ restricts to a unitary $k$-algebra homomorphism from $Z(k_\alpha G)$ to $Z(k_\alpha N_G(P))$.
		\item  For any block idempotent $c$ of $k_\alpha N_G(P)$ and any central idempotent $b$ of $k_\alpha G$, exactly one of $\br_P^{k_\alpha G}(b)c$, $\br_P^{k_\alpha G}(1-b)c$ is nonzero. In particular, there exists a unique block idempotent of $k_\alpha G$ such that $\br_P^{k_\alpha G}(b)c\neq 0$.
	\end{enumerate}
	
\end{proposition}

\begin{proof}
	We are in the context of \ref{void:central extension}, so we can use the notation there. 
	
	(i) By Proposition \ref{prop:restrictions of twisted group algebras}, we have the following commutative diagram of $N_G(P)$-algebras:
	$$\xymatrix{kC_{\hat{G}}(\hat{P})e_R \ar[rr]^{\pi}_{\cong}  \ar@{_{(}->}[d] &  &  k_\alpha C_G(P)\ar@{^{(}->}[d]\\
		k\hat{G}e_R \ar[rr]^{\pi}_{\cong}  &   &  k_\alpha G
	}$$
	Since $kC_{\hat{G}}(\hat{P})e_R\subseteq (k\hat{G}e_R)^P$, by the above commutative diagram we have $k_\alpha C_G(P)\subseteq (k_\alpha G)^P$, whence (i).

	(ii) By Lemma \ref{lemma:basis}, we see that the composition $k C_{\hat{G}}(\hat{P})e_R\hookrightarrow (k\hat{G}e_R)^P\xrightarrow{\br_P^{k\hat{G}}} (k\hat{G}e_R)(P)$ is an isomorphism of $N_G(P)$-algebras. Now statement (ii) follows from the following commutative diagram:
		$$\xymatrix{kC_{\hat{G}}(\hat{P})e_R \ar[rr]^{\pi}_{\cong}  \ar@{_{(}->}[d] &  &  k_\alpha C_G(P)\ar@{^{(}->}[d]\\
		(k\hat{G}e_R)^P \ar[rr]^{\pi}_{\cong}\ar[d]_{\br_P^{k\hat{G}}}  &   &  (k_\alpha G)^P\ar[d]^{\br_P^{k_\alpha G}} \\
			(k\hat{G}e_R)(P) \ar[rr]^{\pi(P)}_{\cong}  &   &  (k_\alpha G)(P)
	}$$
	
	(iii) By definition, it is easy to see that the homomorphism $\br_P^{k_\alpha G}: (k_\alpha G)^P\to k_\alpha N_G(P)$ is unitary. Since $\br_P^{k_\alpha G}: (k_\alpha G)^P\to k_\alpha N_G(P)$ is an $N_G(P)$-algebra homomorphism, it maps $N_G(P)$-fixed points to $N_G(P)$-fixed points, and hence maps $Z(k_\alpha G)=(k_\alpha G)^G$ to $Z(k_\alpha N_G(P))=(k_\alpha N_G(P))^{N_G(P)}$. 
	
	(iv) follows easily from (iii).
\end{proof}

\begin{definition}\label{defi:associated to a block}
	{\rm Let $G$ be a finite group, $P$ a $p$-subgroup of $G$ and $\alpha\in Z^2(G;k^\times)$. Let $b$ be a central idempotent of $k_\alpha G$, $c$ a block of $k_\alpha N_G(P)$ and $U$ an indecomposable $k_\alpha N_G(P)$-module. We say that $c$ (resp. $U$) {\it is associated to} $b$ if $\br_P^{k_\alpha G}(b)c\neq 0$ (resp. $\br_P^{k_\alpha G}(b)U\neq 0$).  By Proposition \ref{prop:Brauer map and central idempotents} (iv), there is a unique block of $k_\alpha G$ to which $c$ (resp. $U$) is associated. Assume that $\alpha\in Z^2(G;\F_p^\times)$. Let $\sigma$ be an element of $\Gamma=\Gal(k/\F_p)$. In this case the Brauer map $\br_P^{k_\alpha G}$ is defined over $\F_p$ (because $\br_P^{k_\alpha G}={\rm Id}_k\otimes\br_P^{(\F_p)_\alpha G}$), hence commuting with $\sigma$. Since $\sigma$ commutes with the Brauer map $\br_P^{k_\alpha G}$, we see that $\sigma(c)$ (resp. ${}^\sigma U$) is associated to $\sigma(b)$ if and only if $c$ (resp. $U$) is associated to $b$. Hence the group $\Gamma_b$ acts on the set of isomorphism classes of indecomposable $k_\alpha N_G(P)$-modules associated to $b$.
	}
\end{definition}

\section{EI-categories $\C$ and partition of weights of $k_\alpha\O_\C$ by blocks of $k_\alpha\C$}\label{s4:Partition of weights}

The following proposition is the reason why we can give a partition of weights of $k_\alpha\O_\C$ by blocks of $k_\alpha\C$ for an EI-category $\C$.

\begin{proposition}\label{prop:eb is in Ge}
	Let $R$ be a commutative ring, $\C$ a finite category, $\alpha\in Z^2(\C;R^\times)$ and $b$ a central idempotent of $R_\alpha\C$. Let $e$ be an idempotent endomorphism of an object $X$ in $\C$ and let $\hat{e}=\alpha(e,e)^{-1}e$. Then $\hat{e}b$ is a central idempotent in the subalgebra $R_\alpha (e\circ\End_\C(X)\circ e)$ of $R_\alpha\C$. In particular, if $\C$ is an EI-category, then $e={\rm Id}_X$ and $\hat{e}b$ is a central idempotent of $R_\alpha G_e$.
\end{proposition}

\begin{proof}
	Since $\hat{e}$ is the unit element of $R_\alpha (e\circ\End_\C(X)\circ e)$, it suffices to show that $\hat{e}b$ is contained in the subalgebra $R_\alpha(e\circ\End_\C(X)\circ e)$. Since $\hat{e}$ is an idempotent in $R_\alpha\C$ and since $b\in Z(R_\alpha\C)$, we have $\hat{e}b=\hat{e}^2b=\hat{e}b\hat{e}$.	Since $\hat{e}$ is an $R$-linear combination of elements of $e\circ\End_\C(X)\circ e$, by the definition of the multiplication in $R_\alpha\C$, we see that $\hat{e}b\hat{e}$ is an $R$-linear combination of elements of $e\circ\End_\C(X)\circ e$. Hence $\hat{e}b\in R_\alpha(e\circ\End_\C(X)\circ e)$, proving the first statement. If $\C$ is an EI-category, then ${\rm Id}_X$ is the unique idempotent in $\End_\C(X)$ and we have $e\circ\End_\C(X)\circ e=G_{{\rm Id}_X}$. This completes the proof.
\end{proof}

\begin{proposition}\label{proposition:varphi(hat(e)b)=hat(f)b}
		Let $\C$ a finite EI-category, $\alpha\in Z^2(\C;k^\times)$ and $b$ a central idempotent of $k_\alpha\C$. Let $X$ and $Y$ be objects in $\C$. Set $e=\Id_X$, $f=\Id_Y$, $\hat{e}=\alpha(e,e)^{-1}e$ and $\hat{f}=\alpha(f,f)^{-1}f$. Assume that $e$ is isomorphic to $f$. Then there exist $s\in f\circ \Hom_\C(X,Y)\circ e$ and $t\in e\circ \Hom_\C(Y,X)\circ f$ satisfying $t\circ s=e$ and $s\circ t=f$. Let $\varphi$ be the $k$-algebra isomorphism $k_\alpha G_e\cong k_\alpha G_f$ described in Proposition \ref{prop:isomorphic idempotents} (i). Then $\varphi(\hat{e}b)=\hat{f}b$.	
\end{proposition}	

\begin{proof}
Since $\C$ is finite and since any finite subset of $k=\bF_p$ is contained in a finite subfield of $k$, we may assume that $\alpha\in Z^2(\C;R^\times)$ for a finite subfield $R$ of $k$. By the commutative diagram in Proposition \ref{prop:commutative diagram for isomorphisms between idempotents} (ii), we may assume that $\alpha$ is trivial.	Then $\varphi$ is the $k$-algebra isomorphism $kG_e\cong kG_f$ sending $x\in kG_e$ to $s x t$. Since $b$ is in the center of $k\C$, we have $\varphi(eb)=sebt=setb=fb$, whence the result.
\end{proof}

\begin{proposition}[a generalisation of Proposition \ref{proposition:varphi(hat(e)b)=hat(f)b}]\label{prop:prove b-weights are well-defined}
Let $\C$ be a finite EI-category, $\alpha\in Z^2(\O_\C;k^\times)$ and $b$ a central idempotent of $k_\alpha\C$. Let $X$ and $Y$ be objects in $\C$. Set $e=\Id_X$, $f=\Id_Y$, $\hat{e}=\alpha(e,e)^{-1}e$ and $\hat{f}=\alpha(f,f)^{-1}f$. Assume that $e$ is isomorphic to $f$. Then there exist $s\in f\circ \Hom_\C(X,Y)\circ e$ and $t\in e\circ \Hom_\C(Y,X)\circ f$ satisfying $t\circ s=e$ and $s\circ t=f$. Let $\varphi$ be the $k$-algebra isomorphism $k_\alpha G_e\cong k_\alpha G_f$ described in Proposition \ref{prop:isomorphic idempotents} (i). Let $P$ be a $p$-subgroup of $G_e$ and let $Q=s\circ P\circ t$. Then $Q$ is a $p$-subgroup of $G_f$. By Proposition \ref{propositon:NGe(eP)/eP}, $\varphi$ restricts to a $k$-algebra isomorphism $k_\alpha N_{G_e}(P)\cong k_\alpha N_{G_f}(Q)$. Recall from Proposition \ref{prop:Brauer map and central idempotents} that $\br_P^{k_\alpha G_e}$ denotes the following composition of $N_{G_e}(P)$-algebra homomorphisms 
$$(k_\alpha G_e)^P\to (k_\alpha G_e)(P)\cong k_\alpha C_{G_e}(P)\hookrightarrow k_\alpha N_{G_e}(P).$$
Similarly, $\br_P^{k_\alpha G_f}$ denotes an $N_{G_f}(Q)$-algebra isomorphism $(k_\alpha G_f)^Q \to k_\alpha N_{G_e}(P)$. We have $\varphi(\br_P^{k_\alpha G_e}(\hat{e}b))=\br_P^{k_\alpha G_f}(\hat{f}b)$.
\end{proposition}	

\begin{proof}
Recall that $k_\alpha G_e$ is a $G_e$-algebra and $k_\alpha G_f$ is a $G_f$-algebra; see \ref{void:central extension}. Denote by $\varphi_0$ the group isomorphism $G_e\cong G_f$ sending $x\in G_e$ to $s\circ x\circ t$. Through the group isomorphism $\varphi_0$, we can regard $G_f$-algebras as $G_e$-algebras. Hence $k_\alpha C_{G_f}(Q)$ and $(k_\alpha G_f)(Q)$ are regarded as $N_{G_e}(P)$-algebras.  As we explained before, we may assume that $\alpha\in Z^2(\C;R^\times)$ for a finite subfield $R$ of $k$. By Proposition \ref{prop:commutative diagram for isomorphisms between idempotents} (ii) we have the following commutative diagram of $k$-algebras:
$$\xymatrix{k\hat{G}_ee_R\ar[rr]^{\hat{\varphi}}_{\cong} \ar[d]_{\pi_e}^{\cong} &   &   k\hat{G}_fe_R \ar[d]^{\pi_f}_{\cong}   \\
	k_\alpha G_e \ar[rr]^{\varphi}_{\cong}  &    &  k_\alpha G_f
}$$
Through the group isomorphism $\hat{G}_e\cong \hat{G}_f$ we can regard $\hat{G}_f$-algebras as $\hat{G}_e$-algebras. In particular, we regard the $\hat{G}_f$-algebra $k\hat{G}_fe_R$ as a $\hat{G}_e$-algebra. Then one easily checks by definition that the isomorphism $\hat{\varphi}:k\hat{G}_ee_R\to k\hat{G}_fe_R$ is an isomorphism of $\hat{G}_e$-algebras. By Proposition \ref{prop:isomorphic as G-algebras} we can regard $k\hat{G}_ee_R$ as a $G_e$-algebra and reagrd $k\hat{G}_f$ as a $G_f$-algebra. Therefore, we can regard $k\hat{G}_f$ as a $G_e$-algebra via the group isomorphism $\varphi_0:G_e\cong G_f$.  Denote by $\tau_e$ (resp. $\tau_f$) the canonical surjection $\hat{G}_e\to G_e$ (resp. $\hat{G}_f\to G_f$) sending $(x,r)$ (resp. $(y,r)$) to $x$ (resp. $y$) for any $x\in G_e$ (resp. $y\in G_f$) and $r\in R^\times$. Then using definiton one also easily checks the following commutative diagram of groups:
$$\xymatrix{ \hat{G}_e\ar[rr]^{\hat{\varphi}}_{\cong}\ar[d]_{\tau_e}  &  &  \hat{G}_f\ar[d]^{\tau_f}\\
	G_e\ar[rr]^{\varphi_0}_{\cong}  &   &   G_f
}$$
Hence the isomorphism $\hat{\varphi}:k\hat{G}_ee_R\to k\hat{G}_fe_R$ is an isomorphism of $G_e$-algebras. Since $\pi_e$ and $\pi_f$ are isomorphisms of $G_e$-algebras (see Proposition \ref{prop:isomorphic as G-algebras}), the isomorphism $\varphi:k_\alpha G_e\cong k_\alpha G_f$ is an isomorphism of $G_e$-algebras.

Then applying the $P$-Brauer functors we have the following commutative diagram of $N_{G_e}(P)$-algebras:
$$\xymatrix{k_\alpha N_{G_e}(P)\ar[rr]^{\varphi}_{\cong}   &   &   k_\alpha N_{G_f}(Q)  \\
k_\alpha C_{G_e}(P)\ar[rr]^{\varphi}_{\cong} \ar@{_{(}->}[d]  \ar@{^{(}->}[u]\ar@/_4pc/[dd]_{\cong} &   &   k_\alpha C_{G_f}(Q) \ar@{^{(}->}[d] \ar@{_{(}->}[u] \ar@/^4pc/[dd]^{\cong}\\
(k_\alpha G_e)^P\ar[rr]^{\varphi}_{\cong} \ar[d]^{\br_P^{k_\alpha G_e}} &   &  (k_\alpha G_f)^Q \ar[d]_{\br_Q^{k_\alpha G_f}}   \\
	(k_\alpha G_e)(P) \ar[rr]^{\varphi(P)}_{\cong} &    &  (k_\alpha G_f)(Q) 
}$$
By Proposition \ref{proposition:varphi(hat(e)b)=hat(f)b}, we have $\varphi(\hat{e}b)=\hat{f}b$. Hence by the above commutative diagram we have 
$$\varphi(P)(\br_P^{k_\alpha G_e}(\hat{e}b))=\br_P^{k_\alpha G_f}(\hat{f}b).$$
Since we identify $(k_\alpha G_e)(P)$ (resp. $(k_\alpha G_f)(Q)$) and $k_\alpha C_{G_e}(P)$ (resp.  $k_\alpha C_{G_f}(Q)$), by the above commutative diagram, we have $\varphi(\br_P^{k_\alpha G_e}(\hat{e}b))=\br_P^{k_\alpha G_f}(\hat{f}b)$.
\end{proof}

\begin{definition}\label{defi:partition of weights by blocks}
	{\rm Keep the notation of Remark \ref{lemma:weights of orbit category} (i). Assume further that $\C$ is an EI-category. Then $e={\rm Id}_X$.  Set $\hat{e}=\alpha(e,e)^{-1}e$. We regard $T$ as a simple $k_\alpha N_{G_e}(P)$-module via the canonical surjection $k_\alpha N_{G_e}(P)\to k_\alpha N_{G_e}(P)/P$. Let $b$ be a central idempotent of $k_\alpha \C$.  Then by Proposition \ref{prop:eb is in Ge}, $\hat{e}b$ is in $k_\alpha G_e$. We say that $(\bar{e},T)$ is a {\it $b$-weight} of $k_\alpha\O_\C$ if $T$ is associated to the central idempotent $\hat{e}b$ of $k_\alpha G_e$; by Definition \ref{defi:associated to a block} this means that $\br_P^{k_\alpha G_e}(\hat{e}b)T\neq 0$. If the pair $(\bar{f},T')$ is another weight of $k_\alpha\O_\C$ isomorphic to $(\bar{e},T)$, then by Proposition \ref{prop:prove b-weights are well-defined} we see that $(\bar{e},T)$ is a $b$-weight if and only if $(\bar{f},T')$ is a $b$-weight. Hence $b$ also gives a partition of isomorphism classes of weights of $k_\alpha\O_\C$. We denote by $\W(k_\alpha\O_\C,b)$ the set of isomorphism classes of $b$-weights; clearly we have $\W(k_\alpha\O_\C,1)=\W(k_\alpha\O_\C)$.
	Let $\sigma\in \Gamma$. Assume that $(\bar{e}, T)$ is a $b$-weight of $k_\alpha\O_\C$, or equivalently, $T$ is associated to the central idempotent $\hat{e}b$ of $k_\alpha G_e$. Then by Definition \ref{defi:associated to a block}, ${}^\sigma T$ is associated to the central idempotent $\sigma(\hat{e}b)=\hat{e}\sigma(b)$ of $k_\alpha G_e$. This implies that $(\bar{e}, {}^\sigma T)$ is a $\sigma(b)$-weight of $k_\alpha\O_\C$. Hence the action of $\Gamma$ on the set of (isomorphism classes of) weights of $k_\alpha\O_\C$ restricts to an action of $\Gamma_b$ on the set of (isomorphism classes of) $b$-weights of $k_\alpha\O_\C$. 
	}	
\end{definition}

\begin{remark}
	{\rm Definition \ref{defi:partition of weights by blocks} gives a partition of (isomorphism classes of) weights of $k_\alpha\O_\C$ by a central idempotent (and hence a block) of $k_\alpha\C$ for any finite EI-category $\C$. If $\C$ is not EI, then we don't know whether $\hat{e}b$ lies in $k_\alpha G_e$, hence we can no longer have such a partition.
	}
\end{remark}

\begin{theorem}[{a block-theoretic refinement of Theorem \ref{theo:bijection between simple moudles and pairs}}]\label{theo:blockwise bijection between simple moudles and pairs}
Let $R$ be a commutative ring, $\C$ a finite EI-category, $\alpha\in Z^2(\C;R^\times)$ and $b$ a central idempotent of $k_\alpha\C$. The map sending a simple $R_\alpha\C b$-module $S$ to the pair $(e,eS)$, where $e$ is an idempotent endomorphism in $\C$ minimal with respect to $eS\neq 0$, induces a bijection $\Pi_b$ between $\S(R_\alpha \C b)$ and the set of isomorphism classes of pairs $(e,T)$ consisting of an idempotent endomorphism $e$ in $\C$ and a simple $R_\alpha G_e(b\hat{e})$-module $T$, where $\hat{e}=\alpha(e,e)^{-1}e$.	
\end{theorem}

\begin{proof}
Let $e$ be an idempotent endomorphism in $\C$. Since $\C$ is EI, $e={\rm Id}_X$ for some object $X$ in $\C$. By Proposition \ref{prop:eb is in Ge}, $\hat{e}b$ is a central idempotent in $R_\alpha G_e=R_\alpha G_e \hat{e}$. If $S$ is a simple $R_\alpha\C b$-module, then $eS$ is an $R_\alpha G_e(\hat{e}b)$-module. Since $eS$ is also a simple $R_\alpha G_e$-module (see Theorem \ref{theo:bijection between simple moudles and pairs}), it is a simple $R_\alpha G_e(\hat{e}b)$-module. Hence the map $\Pi$ in Theorem \ref{theo:bijection between simple moudles and pairs} restricts to a map $\Pi_b$ from $\S(R_\alpha \C b)$ to the set of isomorphism classes of pairs $(e,T)$ consisting of an idempotent endomorphism $e$ in $\C$ and a simple $R_\alpha G_e(b\hat{e})$-module $T$. Since 
$$\S(R_\alpha \C)=\S(R_\alpha\C b)\bigsqcup \S(R_\alpha\C (1-b))$$
and 
$$\S(R_\alpha G_e)=\S(R_\alpha G_e\hat{e})=\S(R_\alpha G_e (\hat{e}b))\bigsqcup \S(R_\alpha G_e (\hat{e}(1-b))),$$
the union of the maps $\Pi_b$ and $\Pi_{1-b}$ is exactly the bijection $\Pi$ in Theorem \ref{theo:bijection between simple moudles and pairs}. Hence $\Pi_b$ is a bijection.
\end{proof}

\begin{proposition}\label{prop:pi_b commutes with gamma_b}
Keep the notation of Theorem \ref{theo:blockwise bijection between simple moudles and pairs}. Assume that $R=k$ and $\alpha\in Z^2(\C;\F_p^\times)$. Then the bijection $\Pi_b$ commutes with the action of the Galois group $\Gamma_b$.
\end{proposition}

\begin{proof}
Since $\Pi_b$ is a part of the bijection $\Pi$ in Theorem \ref{theo:bijection between simple moudles and pairs}, the statement follows from Proposition \ref{prop:pi commutes with gamma}.
\end{proof}

\begin{proposition}\label{prop: Lambda and Omega compatible with blocks}
Let $\C$ be a finite EI-category and $\alpha\in Z^2(\O_\C;k^\times)$. Let $X$ and $Y$ be objects in $\C$. Let $e=\Id_X$ and $f=\Id_Y$ be isomorphic idempotents. Let $b$ be a central idempotent of $k_\alpha\C$. Set $\hat{e}=\alpha(e,e)^{-1}e$ and $\hat{f}=\alpha(f,f)^{-1}f$. Then the bijection $\Lambda_{e,f}$ defined in Notation \ref{notation:lambda_ef} restricts to a $\Gamma_b$-equivariant bijection $\S(k_\alpha G_e \hat{e}b)\to \S(k_\alpha G_f\hat{f}b)$ and the bijection $\Omega_{e,f}$ defined in Proposition \ref{proposition:bijection between isomorphic weights commutes with Galois} (ii) restricts to a $\Gamma_b$-equivariant bijection $\W(k_\alpha\O_{G_e}, \hat{e}b)\to \W(k_\alpha \O_{G_f},\hat{f}b)$. 
\end{proposition}

\begin{proof}
The first statement follows from Proposition \ref{proposition:varphi(hat(e)b)=hat(f)b} and the second from Proposition \ref{prop:prove b-weights are well-defined}.
\end{proof}

\begin{proof}[Proof of Theorem \ref{theorem:main2}]
 We will use the notation of Lemma \ref{lemma: L14 Lemma 3.5}. By Theorem \ref{theo:blockwise bijection between simple moudles and pairs}, there is a bijection
	$$\Pi_b: \S(k_\alpha \C b)\to \bigsqcup_{e\in \E} \S(k_\alpha G_e (\hat{e}b)),$$
	where $\hat{e}=\alpha(e,e)^{-1}e$.
	By Proposition \ref{prop:pi_b commutes with gamma_b}, $\Pi_b$ is commuting with the action of $\Gamma_b$. By assumption, the BGAWC holds for $k_\alpha G_e (\hat{e}b)$, that is, for any $e\in \E$ there is a $\Gamma_b$-equivariant bijection 
	$$\S(k_\alpha G_e (\hat{e}b))\to \bigsqcup_{P\in \X_e}\mathcal{Z}(k_\alpha N_{G_e}(P)/P(\overline{\br_P^{k_\alpha G_e}(\hat{e}b)})),$$
	where $\overline{\br_P^{k_\alpha G_e}(\hat{e}b)}$ is the image of $\br_P^{k_\alpha G_e}(\hat{e}b)\in k_\alpha N_{G_e}(P)$ in $k_\alpha N_{G_e}(P)/P$, and $$\mathcal{Z}(k_\alpha N_{G_e}(P)/P (\overline{\br_P^{k_\alpha G_e}(\hat{e}b)}))$$ is the set of isomorphism classes of projective simple $k_\alpha N_{G_e}(P)/P(\overline{\br_P^{k_\alpha G_e}(\hat{e}b)})$-modules. Hence there is a bijection
	$$\S(k_\alpha \C b)\to \bigsqcup_{e\in\E}\bigsqcup_{P\in \X_e} \mathcal{Z}(k_\alpha N_{G_e}(P)/P(\overline{\br_P^{k_\alpha G_e}(\hat{e}b)}))$$
	commuting with the action of $\Gamma_b$. 
	
	It remains to show that the right side corresponds to a set of representatives of the isomorphism classes of $b$-weights of $k_\alpha \O_\C$. In this double union, $e$ runs over $\E$ and $P$ over $\X_e$. By Lemma \ref{lemma: L14 Lemma 3.5} (ii), this implies that the triples $(e,P,P)$ runs over a set of representatives of the isomorphism classes of idempotent endomorphisms in $\T_\C$. By Lemma \ref{lemma:maximal subgroups of orbit category} (ii), the images of the triples in the morphism sets of $\O_\C$ runs over a set of representatives of the isomorphism classes of idempotent endomorphisms in $\O_\C$. By Lemma \ref{lemma:maximal subgroups of orbit category} (iv), the maximal subgroup determined by the image of any such $(e,P,P)$ in $\O_\C$ is $N_{G_e}(P)/P$, and hence when $[S]$ runs over $\mathcal{Z}(k_\alpha N_{G_e}(P)/P(\overline{\br_P^{k_\alpha G_e}(\hat{e}b)}))$, the quadruple $(e,P,P,S)$ runs over a set of representatives of the isomorphism classes of $b$-weights of $k_\alpha\O_\C$ associated with the image of $(e,P,P)$ in $\O_\C$.  Therefore, the double union corresponds to a set of representatives the isomorphism classes of $b$-weights of $k_\alpha \O_\C$.
\end{proof}

\begin{remark}
	{\rm In Theorem \ref{theorem:main2}, the assumption that $\alpha\in Z^2(\O_\C;\F_p^\times)$ is only using for ensuring that we can consider the actions of $\Gamma_b$. If we consider the non-Galois version of the blockwise Alperin weight conjecture (BAWC), then we can assume that $\alpha\in Z^2(\O_\C;k^\times)$. The proof of Theorem \ref{theorem:main2} can be easily adapted to a proof of the statement that:  if for any idempotent endomorphism $e$ in $\C$, there is a bijection $\S(k_\alpha G_e \hat{e}b)\to \W(k_\alpha \O_{G_e},\hat{e}b)$, then there exists a  bijection $\S(k_\alpha \C b)\to \W(k_\alpha\O_{\C},b)$.  In particular, the BAWC for finite twisted EI-categorie algebras and finite twisted group algebras are equivalent.
	}
\end{remark}

For the proof of Theorem \ref{theorem:main2}, we have developed many auxiliary results. Using part of those auxiliary results, we show the BGAWC for finite twisted EI-category algebras and finite ordinary EI-category algebras are equivalent:

\begin{theorem}\label{theorem:BGAWC for twisted algebra and ordinary algebra are equivalent}
Let $\C$ be a finite EI-category and $\alpha\in Z^2(\O_\C;\F_p^\times)$. Let $\hat{\C}$ be the extension of $\C$ by $\F_p^\times$ associated with $\alpha$; see Notation \ref{notation: extension of category}. The following are equivariant:
\begin{enumerate}[{\rm (i)}]
	\item For any central idempotent $b$ of $k\hat{\C}e_{\F_p}$, there is a $\Gamma_b$-equivariant bijection $\S(k\hat{\C}b)\to \W(k\O_{\hat{\C}},b)$. Here $e_{\F_p}=\frac{1}{|\F_p^\times|}\sum_{r\in \F_p^\times}\sum_{X\in \Ob(\C)}r^{-1}(\Id_X,\alpha(\Id_X,\Id_X)^{-1}r)$ is a central idempotent of $k\hat{\C}$; see Proposition \ref{prop: isomorphism between khatC and kalphaC} (ii).
	\item For any central idempotent $b$ of $k_\alpha\C$, there is a $\Gamma_b$-equivariant bijection $\S(k_\alpha\C b)\to \W(k_\alpha\O_\C,b)$.
\end{enumerate}
\end{theorem}

\begin{proof}[(Sketch of) Proof]
By Proposition \ref{prop: isomorphism between khatC and kalphaC} (ii), we have a $k$-algebra isomorphism $\pi:k\hat{\C}e_{\F_p}\cong k_\alpha\C$. Let $b$ be a central idemptent of $k\hat{\C}e_{\F_p}$. Since $\alpha$ is defined over $\F_p^\times$, the isomorphism $\pi$ is defined over $\F_p$ as well. Then one easily checks that $\pi$ induces a $\Gamma_b$-equivariant bijection between $\S(k\hat{\C}b)$ and $\S(k_\alpha\C \pi(b))$ sending $[T]$ to $[{}_\pi T]$ for any simple $k\hat{\C}b$-module. By using Proposition \ref{prop:restrictions of twisted group algebras} and the commutative diagram of $G_e$-algebras
$$\xymatrix{k\hat{G}_ee_R\ar[rr]^{\hat{\varphi}}_{\cong} \ar[d]_{\pi_e}^{\cong} &   &   k\hat{G}_fe_R \ar[d]^{\pi_f}_{\cong}   \\
	k_\alpha G_e \ar[rr]^{\varphi}_{\cong}  &    &  k_\alpha G_f
}$$
in the proof of Proposition \ref{prop:prove b-weights are well-defined}, one checks that $\pi$ also induces a bijection between $\W(k\O_{\hat{\C}},b)$ and $\W(k_\alpha\O_\C,\pi(b))$. Again by using the assumption that $\alpha$ is defined over $\F_p^\times$, this biection is $\Gamma_b$-equivariant as well. The claim follows.
\end{proof}

\begin{remark}
{\rm In Theorem \ref{theorem:BGAWC for twisted algebra and ordinary algebra are equivalent}, the assumption that $\alpha\in Z^2(\O_\C;\F_p^\times)$ is only using for ensuring that the bijections $\S(k\hat{\C}b)\to\S(k_\alpha\C \pi(b))$ and $\W(k\O_{\hat{\C}},b)\to\W(k_\alpha\O_\C,\pi(b))$ induced by $\pi$ are $\Gamma_b$-equivariant. If we consider the non-Galois version of the BAWC, then we can assume that $\alpha\in Z^2(\O_\C;k^\times)$. The proof of Theorem \ref{theorem:BGAWC for twisted algebra and ordinary algebra are equivalent} can be easily adapted to a proof of the statement that: there is a bijection $\S(k\hat{\C}b)\to \W(k\O_{\hat{\C}},b)$ if and only if there is a bijection $\S(k_\alpha\C \pi(b))\to \W(k_\alpha\O_\C,\pi(b))$. In particular, the BAWC for finite twisted EI-category algebras and finite ordinary EI-category algebras are equivalent.
}
\end{remark}

\section{Automorphisms of categories and proof of Theorem \ref{theorem:main3}}\label{s:Automorphisms of categories}

Let $\C$ be a finite category and $\FF\in \Aut(\C)$. For any object $X$ in $\C$, $\FF$ defines an isomorphism $\FF_X:\End_\C(X)\to \End_\C(\FF(X))$ of moniods; for shorthand we abusively denote the isomorphism $\FF_X$ by $\FF$. Let $e$ be an idempotent in $\End_\C(X)$. Then $\FF(e)$ is an idempotent in $\End_\C(\FF(X))$. The monoid isomorphism $\FF:\End_\C(X)\to \End_\C({\FF(X)})$ restricts to a group isomorphism $\FF_e: G_e\to G_{\FF(e)}$, which in turn extends to a $k$-algebra isomorphism $kG_e\to kG_{\FF(e)}$ still denoted by $\FF_e$. For any $kG_e$-module $U$, we set ${}_{\FF}U$ to be the $kG_{\FF(e)}$-module ${}_{\FF_e}U$ and set ${}_{\FF}[U]$ to be $[{}_{\FF}U]$.

\begin{proposition}\label{prop: Aut action on pairs}
	Let $\C$ be a finite category and $\FF\in \Aut(\C)$. There is a well-defined action of the automorphism group $\Aut(\C)$ on the set of isomorphism classes of pairs $(e,T)$ consisting of an idempotent endomorphism $e$ in $\C$ and a simple $k G_e$-module $T$ via ${}^{\FF}(e,T)=(\FF(e),{}_{\FF_e} T)$. Sometimes we will also write ${}^\FF(P,T)$ as ${}_\FF(P,T)$. Moreover, $(e,T)$ is a weight if and only if $(\FF(e), {}_{\FF_e} T)$ is a weight.
\end{proposition}

\begin{proof}
By the definition of being isomorphic (see the first paragraph of Section \ref{s2:Twisted category algebras and their idempotent endomorphisms}), it is clear that two idempotent endomorphisms $e$ and $e'$ are isomorphic if and only if $\FF(e)$ and $\FF(e')$ are isomorphic. For the first statement we need to show that if $(e,T)$ is isomorphic to $(e',T')$, then $(\FF(e),{}_{\FF_e}T)$ is isomorphic to $(\FF(e'),{}_{\FF_{e'}}(T'))$.  By Proposition \ref{prop:isomorphic idempotents} (i), there is a $k$-algebra isomorphism $\varphi:k G_e\to k G_{e'}$. The group isomorphisms $\FF_e:G_e\to G_{\FF(e)}$ and $\FF_{e'}:G_{e'}\to G_{\FF(e')}$ induce $k$-algebra isomorphisms $\FF_e:kG_e\to kG_{\FF(e)}$ and $\FF_{e'}:kG_{e'}\to kG_{\FF(e')}$, respectively. Let $\FF_\varphi:=\FF_{e'}\circ \varphi\circ \FF_e^{-1}$.
  We need to show that the isomorphism classes of ${}_{\FF_e}T$ and ${}_{\FF_{e'}}T'$ correspond to each other through the induced isomorphism $\FF_\varphi:kG_{\FF(e)}\to kG_{\FF(e')}$. In other words, we need to show that ${}_{\FF_\varphi}({}_{\FF_{e}}T) \cong {}_{\FF_{e'}}T'$ as $kG_{\FF(e')}$-modules. Since $(e,T)$ is isomorphic to $(e',T')$, there is an isomorphism ${}_\varphi T \cong T'$ of $kG_{e'}$-modules. It follows that 
  $${}_{\FF_\varphi}({}_{\FF_{e}}T)={}_{\FF_\varphi\circ \FF_{e}}T={}_{\FF_{e'}\circ \varphi}T={}_{\FF_{e'}}({}_\varphi T)\cong {}_{\FF_{e'}}T',$$
  as required. This proves the first statement. For the second statement, we need to show that $T$ is a projective $kG_e$-module if and only if ${}_{\FF_e} T$ is a projective $kG_{\FF(e)}$-module. It suffices to show that ${}_{\FF_e} kG_e\cong kG_{\FF(e)}$ as $kG_{\FF(e)}$-modules. Indeed, the map sending $a\in {}_{\FF_e}kG_e$ to $\FF_e(a)\in kG_{\FF(e)}$ is a desired isomorphism.
\end{proof}

\begin{proposition}\label{prop:pi commutes with aut}
	Keep the notation of Theorem \ref{theo:bijection between simple moudles and pairs}. Assume that $R=k$ and $\alpha$ is trivial. Then the bijection $\Pi$ commutes with the action of the automorphism group $\Aut(\C)$.
\end{proposition}

\begin{proof}
	Let $\FF\in \Aut(\C)$ and $S$ a simple $k\C$-module. Denote by $[S]$ the isomorphism class of $S$. Let $e$ be an idempotent endomorphism in $\C$, minimal with respect to $eS\neq 0$. Then one easily checks that ${}_{\FF} S$ is a simple $k\C$-module, $\FF(e)({}_{\FF} S)$ is a $k G_{\FF(e)}$-module, and $\FF(e)({}_{\FF} S)\cong {}_{\FF}(eS)$ as $k G_{\FF(e)}$-modules. Hence by the minimality of $e$, we see that $\FF(e)$ is minimal with respect to $\FF(e)({}_{\FF} S)\neq 0$. Therefore, we have 
	$$\Pi([{}_\FF S])=[(\FF(e),\FF(e)({}_\FF S))]=[(\FF(e),{}_\FF(eS))]={}^{\FF}[(e,eS)]={}^{\FF}(\Pi([S])),$$
	where the notation $[(e,S)]$ denotes the isomorphism class of $(e,S)$, and where the third equality holds by Proposition \ref{prop: Aut action on pairs}. 
\end{proof}

\begin{proposition}\label{prop:pi_b commutes with aut_b}
	Keep the notation of Theorem \ref{theo:blockwise bijection between simple moudles and pairs}. Let $H$ be any subgroup of $\Gamma$. Assume that $R=k$ and $\alpha$ is trivial. Then the bijection $\Pi_b$ commutes with the action of the group $(H\times\Aut(\C))_b$.
\end{proposition}

\begin{proof}
	Since $\Pi_b$ is a part of the bijection $\Pi$ in Theorem \ref{theo:bijection between simple moudles and pairs}, the statement follows from Propositions \ref{prop:pi commutes with gamma} and \ref{prop:pi commutes with aut}.
\end{proof}

\begin{remark}\label{remark: aut and associated to a block}
	{\rm Let $G$ be a finite group, $P$ a $p$-subgroup of $G$. Now we denote by $\br_P^{kG}:(kG)^P\to kC_G(P)$ the map sending $\sum_{g\in G}\alpha_g g$ (where $\alpha_g\in k$) to $\sum_{g\in C_G(P)}\alpha_g g$.
		Let $b$ be a central idempotent of $k G$, $c$ a block of $k N_G(P)$ and $U$ an indecomposable $k N_G(P)$-module. By Definition \ref{defi:associated to a block}, $c$ (resp. $U$) {\it is associated to} $b$ if $\br_P^{k G}(b)c\neq 0$ (resp. $\br_P^{k G}(b)U\neq 0$).  By Proposition \ref{prop:Brauer map and central idempotents} (iv), there is a unique block of $k G$ to which $c$ (resp. $U$) is associated. Let $G'$ be another finite group and $\FF:G\cong G'$ a group isomorphism. Then $\FF$ extends $k$-linearly to a $k$-algebra isomorphism $kG\cong kG'$ which in turn restricts to $k$-algebra isomorphisms $kC_G(P)\cong kC_{G'}(\FF(P))$ and $kN_G(P)\cong kN_{G'}(\FF(P))$; we abusively denote these isomorphisms by $\FF$.  The following diagram is obviously commutative:
		$$\xymatrix{ (kG)^P \ar[rr]^{\br_P^{kG}~~} \ar[d]_{\FF} & & kC_G(P)  \ar@{^{(}->}[r]  \ar[d]^{\FF} &     kN_G(P) \ar[d]^{\FF} \\
			(kG')^{\FF(P)} \ar[rr]^{\br_{\FF(P)}^{kG'}~~~}  & & kC_{G'}(\FF(P))  \ar@{^{(}->}[r] &    kN_{G'}(\FF(P))
		}$$
		Therefore, $\FF(c)$ (resp. ${}_{\FF} U$) is associated to $\FF(b)$ if and only if $c$ (resp. $U$) is associated to $b$. Hence the group $\Aut(G)_b$ acts on the set of isomorphism classes of indecomposable $k N_G(P)$-modules associated to $b$. The isomoprhism $\FF:N_G(P)\cong N_{G'}(\FF(P))$ induces a $k$-algebra isomorphism $N_G(P)/P\cong N_{G'}(\FF(P))/\FF(P)$ still denoted by $\FF$. For any weight $(P,T)$ of $k\O_G$, we set ${}_{\FF}(P,T)$ to be $(\FF(P),{}_{\FF}T)$, which is a weight of $k\O_{G'}$. One easily checks that if $(P,T)$ and $(Q,T')$ are isomorphic weights of $k\O_G$, then $(\FF(P),{}_\FF T)$ and $(\FF(Q),{}_\FF T')$ are isomorphic weights of $k\O_{G'}$; we set ${}_\FF[(P,T)]:=[{}_\FF(P,T)]$. 
	}
\end{remark}

\begin{void}\label{defi: action of aut on weights of orbit category}
{\rm (i) Let $\C$ be a finite category and $\FF\in \Aut(\C)$. Then $\FF$ induces an automorphism $\T_{\FF}$ of $\T_\C$ sending an object $(X,P)$ in $\T_\C$ to $(\FF(X),\FF(P))\in \T_\C$ with the obvious maps induced by $\FF$ on morphism sets. Similarly, $\FF$ induces an automorphism $\O_{\FF}$ of $\O_\C$ sending an object $(X,P)$ in $\O_\C$ to $(\FF(X),\FF(P))\in \O_\C$. By Proposition \ref{prop: Aut action on pairs}, the automorphism $\O_{\FF}$ acts on the set of (isomorphism classes of) weights of $k\O_\C$. Hence we can define an action of $\Aut(\C)$ on the set of (isomorphism classes of) weights of $k\O_\C$ by letting $\FF$ act as $\O_{\FF}$. With the notation of Remark \ref{lemma:weights of orbit category} (i), we see that a weight of $k\O_\C$ is a pair $(\bar{e}, T)$, where $\bar{e}=P\circ e\circ P$ for some idempotent endomorphism $e$ of some object $X$ in $\C$, $P$ is a not necessarily unitary $p$-subgroup of the moniod $\End_\C(X)$, and $T$ is a projective simple $k N_{G_e}(e\circ P)/(e\circ P)$-module. The automorphism $\FF$ of $\C$ induces a group isomorphism 
	$$N_{G_e}(e\circ P)/(e\circ P)\cong N_{G_{\FF(e)}}(\FF(e)\circ \FF(P))/(\FF(e)\circ \FF(P))$$ which in turn induces a $k$-algebra isomorphism 
	$$kN_{G_e}(e\circ P)/(e\circ P)\cong kN_{G_{\FF(e)}}(\FF(e)\circ \FF(P))/(\FF(e)\circ \FF(P));$$
	we again abusively denote these isomorphisms by $\FF$. Write $\FF(\bar{e}):=\FF(P)\circ \FF(e)\circ \FF(P)$. Then by definition, we have
	$${}^{\FF}(\bar{e},T)={}^{\O_{\FF}}(\bar{e},T)=(\FF(\bar{e}),{}_{\FF}T).$$
	
(ii) Assume further that $\C$ is a finite EI-category and $b$ is a central idempotent in $k\C$. Assume that $(\bar{e}, T)$ is a $b$-weight of $k\O_\C$, or equivalently, $T$ is associated to the central idempotent $eb$ of $kG_e$. Then by Remark \ref{remark: aut and associated to a block}, ${}_{\FF} T$ is associated to the central idempotent $\FF(eb)=\FF(e)\FF(b)$ of $kG_{\FF(e)}$. This implies that $(\FF(\bar{e}), {}_{\FF}T)$ is an $\FF(b)$-weight of $k\O_\C$. Let $H$ be any subgroup of $\Gamma$. Then the action of $H\times\Aut(\C)$ on the set of (isomorphism classes of) weights of $k\O_\C$ restricts to an action of $(H\times\Aut(\C))_b$ on the set of (isomorphism classes of) $b$-weights of $k\O_\C$.
}
\end{void}

\begin{proposition}\label{prop:two conditions}
	
	Let $\C$ be a finite category (resp. EI-category) and $b$ the unit element (resp. a central idempotent) of $k\C$.	
	Then for any idempotent endomorphism $e$ in $\C$, $eb$ is a central idempotent in $k G_e$; see Proposition \ref{prop:eb is in Ge} below. Let $H$ be any subgroup of $\Gamma$. Assume that for any idempotent endomorphism $e$ in $\C$, there is a bijection $\Omega_e:\S(k G_e eb)\to \W(k \O_{G_e},eb)$ fulfilling the following conditions:
	\begin{enumerate}[{\rm (i)}]
		\item For any idempotent $e$ in $\C$ and any simple $kG_e eb$-module $T$, $\Omega_{\FF(e)}({}_{(\sigma,\FF)} [T])={}_{(\sigma,\FF)}\Omega_e([T])$ for any $(\sigma,\FF)\in (H\times\Aut(\C))_b$.
		\item For any isomorphic idempotents $e$ and $f$ in $\C$, we have $\Omega_f\circ\Lambda_{e,f}=\Omega_{e,f}\circ \Omega_e$, where $\Lambda_{f,e}$ is defined as in Notation \ref{notation:lambda_ef} and $\Omega_{e,f}$ is defined as in Proposition \ref{proposition:bijection between isomorphic weights commutes with Galois}.
	\end{enumerate}
	
	Then there exists an $(H\times\Aut(\C))_b$-equivariant bijection $\S(k \C b)\to \W(k\O_{\C},b)$.
\end{proposition}

\begin{proof}
 We will use the notation of Lemma \ref{lemma: L14 Lemma 3.5}. By Theorem \ref{theo:blockwise bijection between simple moudles and pairs}, there is a bijection
	$$\Pi_b: \S(k \C b)\to \bigsqcup_{e\in \E} \S(k G_e (eb)).$$
	 In this bijection, the left side is an $(H\times\Aut(\C))_b$-set. The right side is an $(H\times\Aut(\C))_b$-set in the following way: for any $(\sigma,\FF)\in (H\times\Aut(\C))_b$ and any $[T]\in \S(k G_e (eb))$, where $T$ is a simple $kG_e(eb)$-module, we see that ${}_{(\sigma,\FF)} T$ is a simple $kG_{\FF(e)}(\FF(e)b)$-module. Let $e'$ be the unique element in $\E$ such that $\FF(e)\cong e'$ in $\C$.  Then we define ${}^{(\sigma,\FF)}[T]:=\Lambda_{\FF(e),e'}([{}_{(\sigma,\FF)}T])$, and this makes $\bigsqcup_{e\in \E} \S(k G_e (eb))$ into an $(H\times\Aut(\C))_b$-set. By Proposition \ref{prop:pi_b commutes with aut_b}, $\Pi_b$ is commuting with the action of $(H\times\Aut(\C))_b$.

By assumption, for any $e\in \E$ there is a bijection 
$$\Omega_e:\S(k_\alpha G_e (eb))\to \W(k\O_{G_e},eb)=\bigsqcup_{P\in \X_e}\mathcal{Z}(kN_{G_e}(P)/P(\overline{\br_P^{k G_e}(eb)})),$$
where the equality is an identification, $\overline{\br_P^{k G_e}(eb)}$ is the image of $\br_P^{kG_e}(eb)\in k N_{G_e}(P)$ in $k N_{G_e}(P)/P$, and $$\mathcal{Z}(k N_{G_e}(P)/P (\overline{\br_P^{k G_e}(eb)}))$$ is the set of isomorphism classes of projective simple $kN_{G_e}(P)/P(\overline{\br_P^{k G_e}(eb)})$-modules.	
	 Hence there is a bijection
	$$\Sigma_b:=\bigsqcup_{e\in \E}\Omega_e:~\bigsqcup_{e\in \E} \S(k G_e (eb))\to \bigsqcup_{e\in\E}\bigsqcup_{P\in \X_e} \mathcal{Z}(k N_{G_e}(P)/P(\overline{\br_P^{k G_e}(eb)})).$$
As we explained, the left side is an $(H\times\Aut(\C))_b$-set. The right side is an $(H\times\Aut(\C))_b$-set in the following way: for any $(\sigma,\FF)\in (H\times\Aut(\C))_b$, $e\in \E$, $P\in \X_e$ and $[T]\in \mathcal{Z}(k N_{G_e}(P)/P (\overline{\br_P^{k G_e}(eb)}))$ (where $T$ is a projective simple $k N_{G_e}(P)/P (\overline{\br_P^{k G_e}(eb)})$-module), we see that ${}_{(\sigma,\FF)} T$ is a projective simple $k N_{G_{\FF(e)}}(\FF(P))/\FF(P) (\overline{\br_{\FF(P)}^{k G_{\FF(e)}}(\FF(e)b)})$-module. Let $e'\in \E$ be the unique element in $\E$ such that $\FF(e)\cong e'$ in $\C$, and $Q$ the unique element in $\X_{e'}$ such that $(\FF(e),\FF(P),\FF(P))\cong (e',Q,Q)$ as idempotent endomorphisms in $\T_\C$; see Lemma \ref{lemma: L14 Lemma 3.5} (ii) for the existence and uniqueness of $Q$. Then we define $${}^{(\sigma,\FF)}[T]:=\Omega_{\FF(e),e'}([(\FF(\bar{e}),{}_{(\sigma,\FF)}T)])\in \W(k\O_{G_{e'}},e'b)=\bigsqcup_{Q\in \X_{e'}}\mathcal{Z}(kN_{G_{e'}}(Q)/Q(\overline{\br_{Q}^{k G_{e'}}(e'b)})),$$ and this makes 
$$\bigsqcup_{e\in\E}\bigsqcup_{P\in \X_e} \mathcal{Z}(k N_{G_e}(P)/P(\overline{\br_P^{k G_e}(eb)}))$$
into an $(H\times\Aut(\C))_b$-set. 

Next we show that $\Sigma_b$ commutes with $(H\times\Aut(\C))_b$. For any $(\sigma,\FF)\in (H\times\Aut(\C))_b$, any $e\in \E$ and any $[T]\in \S(k G_e (eb))$, let $e'$ be the unique element in $\E$ such that $\FF(e)\cong e'$ in $\C$.  Then  
$$\Sigma_b({}^{(\sigma,\FF)}[T])=\Sigma_b(\Lambda_{\FF(e),e'}([{}_{(\sigma,\FF)}T]))=\Omega_{e'}(\Lambda_{\FF(e),e'}([{}_{(\sigma,\FF)}T]))=\Omega_{\FF(e),e'}(\Omega_{\FF(e)}([{}_{(\sigma,\FF)}T]))$$
$$=\Omega_{\FF(e),e'}({}_{\sigma,\FF)}\Omega_e([T]))={}^{(\sigma,\FF)}(\Omega_e([T]))={}^{(\sigma,\FF)}(\Sigma_b([T])),$$
where the third equality holds by the condition (ii), the fourth equality holds by the condition (i), and the fifth equality holds by the definition of the action of $(H\times\Aut(\C))_b$ on $$\bigsqcup_{e\in\E}\bigsqcup_{P\in \X_e} \mathcal{Z}(k N_{G_e}(P)/P(\overline{\br_P^{k G_e}(eb)})).$$
Hence $\Sigma_b$ commutes with $(H\times\Aut(\C))_b$.

 Now we obtain an $(H\times\Aut(\C))_b$-equivariant bijection
$$\Sigma_b\circ \Pi_b:\S(k \C b)\to \bigsqcup_{e\in\E}\bigsqcup_{P\in \X_e} \mathcal{Z}(k N_{G_e}(P)/P(\overline{\br_P^{k G_e}(eb)}))$$
By the last paragraph of the proof of Theorem \ref{theorem:main2}, the right side corresponds to a set of representatives of the isomorphism classes of $b$-weights of $k \O_\C$. This completes the proof.
\end{proof}

\begin{proposition}\label{prop:Satisfy condition i}
	Let $\C$ be a finite category (resp. EI-category) and $b$ the unit element (resp. a central idempotent) of $k\C$.	
	Then for any idempotent endomorphism $e$ in $\C$, $eb$ is a central idempotent in $k G_e$;  see Proposition \ref{prop:eb is in Ge}. Let $H$ be any subgroup of $\Gamma_b$. Assume that for any idempotent endomorphism $e$ in $\C$ there is an $(H\times\Aut(G_e))_{be}$-equivariant bijection $\Omega'_e:\S(k G_e eb)\to \W(k \O_{G_e},eb)$. Then we can modify the choices of some bijections and obtain a family of $(H\times\Aut(G_e))_{be}$-equivariant bijections $\{\Omega_e:\S(k G_eeb)\to \W(k \O_{G_e},eb)\}$ satisfying the following two conditions in Proposition \ref{prop:two conditions}: 
	
\begin{enumerate}[{\rm (i)}]
	\item 	For any idempotent $e$ in $\C$ and any simple $kG_e eb$-module $T$, $\Omega_{\FF(e)}({}_{(\sigma,\FF)} [T])={}_{(\sigma,\FF)}\Omega_e([T])$ for any $(\sigma,\FF)\in (H\times\Aut(\C))_b$.
	\item  For any objects $X,Y$ in $\C$ and any isomorphic idempotents $e\in \End_\C(X)$, $f\in \End_\C(Y)$, we have $\Omega_f\circ\Lambda_{e,f}=\Omega_{e,f}\circ \Omega_e$, where $\Lambda_{f,e}$ is defined as in Notation \ref{notation:lambda_ef} and $\Omega_{e,f}$ is defined as in Proposition \ref{proposition:bijection between isomorphic weights commutes with Galois}.
\end{enumerate}	
\end{proposition}

\begin{proof}
Let $\I$ be the set of all idempotent endomorphisms in $\C$. Define a relation on the set $\I$ by $e\sim f$, if there exists a group isomorphism $\psi:G_e\cong G_f$ such that the induced $k$-algebra isomorphism $\psi:kG_e\cong kG_f$ satisfying $\psi(eb)=f\sigma^{-1}(b)$ for some $\sigma\in H$. For convenience, we say that the pair $(\sigma,\psi)$ belongs to $(H\times \Iso(G_e,G_f))_b$. One easily checks that the relation $\sim$ is an equivalence relation. By Proposition \ref{proposition:varphi(hat(e)b)=hat(f)b}, if $e$ and $f$ are isomorphic idempotent endomorphisms in $\C$, then $e\sim f$. Also, if $e$ and $f$ are idempotent endomorphisms in $\C$ such that $f=\FF(e)$ for some $(\sigma,\FF)\in (H\times\Aut(\C))_b$, then $e\sim f$.
Let $\I_0\subseteq \I$ be a set of representatives of the equivalence classes of the relation $\sim$. For any $e\in \I_0$, we set $\Omega_e:=\Omega'_e$. Let $f$ be an arbitrary idempotent endomorphism in $\C$. There exists a unique idempotent $e\in \I_0$ such that $e\sim f$. By definition, there is a group isomorphism $\psi:G_e\cong G_f$ with the induced $k$-algebra isomorphism $\psi:kG_e\cong kG_f$ satisfying $\psi(eb)=f\sigma^{-1}(b)$ for some $\sigma\in H$. By Remark \ref{remark: aut and associated to a block}, if $(P,T)$ is a weight of $k\O_{G_e}$, then $(\psi(P),{}_\psi T)$ is a weight of $k\O_{G_f}$. Now we see that if  $(P,T)$ is an $eb$-weight of $k\O_{G_e}$, then $(\psi(P),{}^\sigma({}_\psi T))$ is an $fb$-weight of $k\O_{G_f}$. For any simple $k G_f fb$-module $T$,  we define $\Omega_f([T])$ to be ${}_{(\sigma,\psi)}\Omega_e({}_{(\sigma^{-1},\psi^{-1})}[T]):={}^\sigma({}_\psi\Omega_e({}_{\psi^{-1}}[{}^{\sigma^{-1}}T]))$. Since $\Omega_e$ is $(H\times\Aut(G_e))_b$-equivariant, $\Omega_f$ is independent of the choice of the pair $(\sigma,\psi)$. One easily checks that $\Omega_f$ is a bijection from $\S(k G_f fb)$ to $\W(k \O_{G_f},fb)$. We claim that $\Omega_f$ is $(H\times\Aut(G_f))_{fb}$-equivariant. Let $(\sigma',\FF)\in (H\times\Aut(G_f))_{fb}$. Then $(\sigma^{-1},\psi^{-1})\circ(\sigma',\FF)\circ (\sigma,\psi)\in (H\times\Aut(G_e))_{eb}$.  For any simple $k G_f fb$-module $T$, we have
\begin{align*}
	\begin{split}
		\Omega_f({}_{(\sigma',\FF)}[T])&={}_{(\sigma,\psi)}\Omega_e({}_{(\sigma^{-1},\psi^{-1})\circ (\sigma',\FF)}T)={}_{(\sigma,\psi)}\Omega_e({}_{(\sigma^{-1},\psi^{-1})\circ (\sigma',\FF)\circ(\sigma,\psi)\circ (\sigma^{-1},\psi^{-1})}T)\\
		&={}_{(\sigma,\psi)\circ (\sigma^{-1},\psi^{-1})\circ (\sigma',\FF)\circ(\sigma,\psi)}\Omega_e({}_{(\sigma^{-1},\psi^{-1})}T)={}_{(\sigma',\FF)\circ(\sigma,\psi)}\Omega_e({}_{(\sigma^{-1},\psi^{-1})}T)\\
		&={}_{(\sigma',\FF)}\Omega_f([T]),
	\end{split}
\end{align*}
as claimed. Next we show that the family of bijections $\{\Omega_e:\S(k G_eeb)\to \W(k \O_{G_e},eb)\}$ satisfies the required two conditions:
	
\textbf{Condition (i)}. Now let $e$ be an arbitrary idempotent endomorphism in $\C$. Let $e_0$ be the unique element in $\I_0$ such that $e\sim e_0$ via some pair $(\sigma_0,\psi_0)\in (H\times\Iso(G_{e_0},G_e))_b$.  Let $(\sigma,\FF)$ be an arbitrary element of $(H\times\Aut(\C))_b$. Write $\sigma'=\sigma\circ \sigma_0$ and $\psi'=\FF\circ \psi_0$. Then we have $(\sigma',\psi')\in (H\times \Iso(G_{e_0},G_{\FF(e)}))_b$.  For any simple $kG_e eb$-module $T$, 
\begin{align*}
	\begin{split}
\Omega_{\FF(e)}({}_{(\sigma,\FF)} [T])&={}_{(\sigma',\psi')}\Omega_{e_0}({}_{({\sigma'}^{-1},{\psi'}^{-1})\circ(\sigma,\FF)}T)={}_{(\sigma',\psi')}\Omega_{e_0}({}_{(\sigma_0^{-1},\psi_0^{-1})}T)\\
&={}_{(\sigma',\psi')\circ (\sigma_0^{-1},\psi_0^{-1})\circ (\sigma_0,\psi_0)}\Omega_{e_0}({}_{(\sigma_0^{-1},\psi_0^{-1})}T)\\
&={}_{(\sigma,\FF)}\Omega_e([T])
\end{split}
\end{align*}
	where the first and fourth equalities hold by the definitions of $\Omega_{\FF(e)}$ and $\Omega_e$, respectively. 
	Therefore, the family of bijections $\{\Omega_e:\S(k G_eeb)\to \W(k \O_{G_e},eb)\}$ satisfies condition (i).
	
\textbf{Condition (ii)}. Let us first assume that $e\in \I_0$. Let $s\in f\circ \Hom_\C(X,Y)\circ e$ and $t\in e\circ \Hom_\C(Y,X)\circ f$ satisfying $t\circ s=e$ and $s\circ t=f$. Then we have a group isomorphism $\varphi:G_e\cong G_f$ sending $x\in G_e$ to $s\circ x\circ t$. By Proposition \ref{proposition:varphi(hat(e)b)=hat(f)b}, we have $({\rm Id}_k,\varphi)\in (H\times \Iso(G_e,G_f))_b$. For any simple $k G_f fb$-module $T$, by definition of $\Omega_f$, we have 
$$\Omega_f([T])={}_{(\Id_k,\varphi)}\Omega_e({}_{(\Id_k,\varphi^{-1})}T)={}_{\varphi}\Omega_e({}_{\varphi^{-1}}T)=\Omega_{e,f}\circ\Omega_e\circ \Lambda_{f,e}([T])$$
where the third equality holds by the definitions of $\Lambda_{f,e}$ and $\Omega_{e,f}$. Now let us remove the assumption that $e\in \I_0$. Since $e$ and $f$ are isomorphic idempotents, there is a unique idempotent $d\in \I_0$ such that $e\sim d\sim f$. Then by the argument above, we have $\Omega_e:=\Omega_{d,e}\circ \Omega_d\circ \Lambda_{e,d}$ and $\Omega_f:=\Omega_{d,f}\circ \Omega_d\circ \Lambda_{f,d}$. It follows that 
$$\Omega_f=\Omega_{d,f}\circ \Omega_{e,d}\circ \Omega_e\circ \Lambda_{d,e} \circ \Lambda_{f,d}=\Omega_{e,f}\circ \Omega_e\circ \Lambda_{f,e}.$$ 
Equivalently, $\Omega_f\circ\Lambda_{e,f}=\Omega_{e,f}\circ \Omega_e$.
Therefore, condition (ii) holds.
\end{proof}

\begin{proof}[Proof of Theorem \ref{theorem:main3}]
This follows from Propositions \ref{prop:two conditions} and \ref{prop:Satisfy condition i}.
\end{proof}

\bigskip\noindent\textbf{Acknowledgements.}\quad When writing this paper, the author is a visitor at City St George's, University of London supported by China Scholarship Council (202506770066) from 2026 to 2028. The author would like to thank Prof. Markus Linckelmann for some very helpful discussions and thank City for its hospitality and comfortable working environment. The author acknowledges support from National Natural Science Foundation of China (12471016), China Postdoctoral Science Foundation (GZC20262006, 2025T001HB), and Fundamental Research Funds for the Central Universities (CCNU24XJ028).

\bigskip
{\footnotesize School of Mathematics and Statistics, Central China Normal University, Wuhan 430079, China 
	
	Email address: xinhuang@mails.ccnu.edu.cn}

\end{document}